%% file: BeECH_v2.tex
\documentclass{birkjour} 
\usepackage[utf8]{inputenc}
\usepackage{amsmath, amsthm, amssymb}
\usepackage{enumerate}
\usepackage{mathrsfs}
\usepackage{mathtools}
\usepackage{tikz-cd}
\usepackage{xcolor}
\usepackage{graphicx}

\newtheorem{theorem}{Theorem}[section]

\newtheorem{proposition}[theorem]{Proposition}

\newtheorem{lemma}[theorem]{Lemma}

\newtheorem*{problem*}{Problem}
\newtheorem*{theorem*}{Theorem}
\newtheorem*{proposition*}{Proposition}
\newtheorem*{corollary*}{Corollary}
\newtheorem*{lemma*}{Lemma}
\newtheorem*{claim*}{Claim}
\newtheorem*{solution*}{Solution}

\theoremstyle{definition}
\newtheorem{definition}[theorem]{Definition}
\newtheorem{example}[theorem]{Example}
\newtheorem{remark}[theorem]{Remark}
\newtheorem*{remark*}{Remark}
\newtheorem*{example*}{Example}
\newtheorem{claim}{Claim}[theorem]
\newtheorem*{warning*}{Warning}

\newcommand{\R}{{\mathbb R}}

\newcommand{\N}{{\mathbb N}}
\newcommand{\C}{{\mathbb C}}
\newcommand{\Z}{{\mathbb Z}}

\usepackage{graphicx}
\usepackage{array}
\usepackage{float}

\newcommand{\bbN}{\mathbb{N}}

\newcommand{\bbR}{\mathbb{R}}
\newcommand{\bbQ}{\mathbb{Q}}
\newcommand{\bbC}{\mathbb{C}}

\definecolor{pink}{RGB}{255,105,180}
\definecolor{green}{RGB}{76,187,23}

\usepackage{hyperref}

\usepackage{subcaption}

\definecolor{indigo}{RGB}{51,0,102}
\definecolor{brightpurple}{RGB}{102,0,153}
\definecolor{fuchsia}{RGB}{180,51,180}
\definecolor{jolightpurple}{RGB}{153,51,255}

\hypersetup{colorlinks,
linkcolor=brightpurple,
filecolor=brightpurple,
urlcolor=indigo,
citecolor=fuchsia}

\title[Symplectic embeddings of polydisks into half integer ellipsoids]{Symplectic embeddings of four-dimensional polydisks into half integer ellipsoids}
\author{L. Digiosia, J. Nelson, H. Ning, M. Weiler, Y. Yang}

\begin{document}

\maketitle

\begin{abstract} 
We obtain new sharp obstructions to symplectic embeddings of four-dimensional polydisks $P(a,1)$ into four-dimensional ellipsoids $E(bc,c)$ when $1\le a< 2$ and $b$ is a half-integer.  { When $1 \leq a < 2-O(b^{-1})$ we demonstrate that} $P(a,1)$ symplectically embeds into $E(bc,c)$ if and only if $a+b\le bc$.  Our results show that inclusion is optimal and extend the result by Hutchings \cite{H} when $b$ is an integer.  Our proof is based on a combinatorial criterion developed by Hutchings \cite{H} to obstruct symplectic embeddings.   
 
\end{abstract}

\tableofcontents

	\section{Introduction}

Embedded contact homology (ECH) is useful for obtaining obstructions of symplectic embeddings of four dimensional toric domains with contact type boundary.  When considering embeddings of poydisks into balls and ellipsoids, ECH capacities do not tend to give very good obstructions.  Hutchings introduced a combinatorial criterion using embedded contact homology in a more refined way than ECH capacities \cite{H}, which has led to considerable progress in understanding the problem of symplectically embedding one convex toric domain into another.  This criterion is stated in terms of convex generators, which are integral paths in the plane whose vertices are lattice points with additional restrictions on the edges.  Hutchings showed that the boundary of any convex toric domain can be perturbed so that the generators of the ECH chain complex correspond to convex generators (for an induced contact form and up to large symplectic action).  The generators of ECH are multisets of closed periodic orbits of the Reeb vector field associated to the contact form. 

By studying the combinatorial data more closely, in Section \ref{bigsection2} we are able to obtain new sharp obstructions to symplectic embeddings of four-dimensional polydisks $P(a,1)$ into four-dimensional ellipsoids $E(bc,c)$ when $1\le a< 2$ and $b$ is a half-integer $p/2$.  When $1 \leq a \leq 2-\varepsilon$ and $p>O(\varepsilon^{-1})$ we demonstrate that $P(a,1)$ symplectically embeds into $E(bc,c)$ if and only if $a+b\le bc$.  Our results show that inclusion is optimal and extend the result by Hutchings  when $b$ is an integer \cite[Thm.~1.5]{H}. 

In Section \ref{bigsection3} we elucidate the combinatorial subtleties of natural classes of minimal convex generators associated to ellipsoids in applying the Hutchings criterion.  We classify all minimal convex generators associated to an ellipsoid, providing a converse to \cite[Lem.~2.1(a)]{H}, and demonstrate the limitations of the simplest, and often most natural, class of convex generators used in the Hutchings criterion to provide obstructions to symplectic embeddings.  We provide further computations of the combinatorial ECH index associated to minimal convex generators, which may provide a future means of obtaining ECH based obstructions to symplectic embeddings.

	\subsection{Embeddings of four-dimensional polydisks into ellipsoids}

In this paper we investigate the question of when one convex toric symplectic four-manifold can be symplectically embedded into another.  In particular, we demonstrate that inclusion is optimal for symplectic embeddings of four-dimensional polydisks $P(a,1)$ into ellipsoids $E(bc,c)$ {when $1\le a< 2-\varepsilon$ and $b$ is a half-integer $p/2$ such that $p>O(\varepsilon^{-1})$}.

Four dimensional toric domains are defined as follows:
	
	\begin{definition}
		Let $\Omega$ be a domain in the first quadrant of the plane $\bbR^2$. The \emph{toric domain} $X_\Omega$ associated to $\Omega$ is defined  to be
		\[
		X_\Omega = \left\{(z_1, z_2) \in \bbC^2\, \left| \, (\pi|z_1|^2, \pi|z_2|^2)\in  \Omega \right. \right\},
		\]
		with the restriction of the standard symplectic form on $\bbC^2$, namely
		\[
		\omega = \sum_{i = 1}^2dx_i\wedge dy_i.
		\]
			When the symplectic form is understood we drop it from our notation. 
	\end{definition}
		
		In addition, if
		\[
		\Omega = \{(x, y) \, | \, 0 \le x \le A, 0 \le y \le f(x)\},
		\]
		where $f: [0, A] \to [0, \infty)$ is a nonincreasing function, then we say $X_\Omega$ (and $\Omega$) is \emph{convex} if $f$ is concave, and $X_\Omega$ (and $\Omega$) is \emph{concave} if $f$ is convex with $f(A) = 0$.

\begin{example}
	If $\Omega$ is the triangle with vertices $(0, 0), (a, 0)$, and $(0, b)$, then $X_\Omega$ is the ellipsoid
	\[
	E(a, b) = \left\{(z_1, z_2) \in \bbC^2\, \left| \, \frac{\pi|z_1|^2}{a}+ \frac{\pi|z_2|^2}{b} \le 1 \right. \right\}.
	\]
	Note that when $a = b$, we obtain the closed four-ball $B(a) = E(a, a)$. An ellipsoid is both a convex and a concave toric domain. If $\Omega$ is the rectangle with vertices $(0, 0), (a, 0), (0, b)$, and $(a, b)$, then $X_\Omega$ is the polydisk 
	\[
	P(a, b) = \{(z_1, z_2) \in \bbC^2\, \left| \, \pi|z_1|^2 \le a, \pi|z_2|^2\le b\right. \},
	\]
	which is a convex toric domain with $f(x) \equiv b$ on the interval $[0,a]$. 
\end{example}	
In dimension four, substantial progress has been made on understanding the nature of symplectic embeddings between symplectic four-manifolds with contact type boundary by way of embedded contact homology (ECH). Hutchings used ECH to define the \emph{ECH capacities} of any symplectic four-manifold  in \cite{H2}. The ECH capacities of $(X,\omega)$ are a sequence of real numbers $c_k$ with
	\[
	0 = c_0(X,\omega) < c_1(X,\omega) \le c_2(X,\omega) \le \cdots \le \infty,
	\]
	such that if $(X,\omega)$ symplectically embeds into $(X',\omega')$ then
	\begin{equation}\label{eqn:capineq}
	c_k(X,\omega) \le c_k(X',\omega') \ \mbox{ 	for all $k$. }
	\end{equation}
Choi, Cristofaro-Gardiner, Frenkel, Hutchings, and Ramos \cite{CHO} computed ECH capacities of all concave toric domains, yielding  sharp obstructions to certain symplectic embeddings involving concave toric domains. Cristofaro-Gardiner \cite{CG} proved that the ECH capacities give sharp obstructions to all symplectic embeddings of concave toric domains into convex toric domains, generalizing the results of McDuff \cite{dusaellipsoid2, MC} and Frenekl-Müller \cite{FR}.
		
\begin{remark}\label{rmk:echcapP}		
When studying symplectic embeddings of convex toric domains, such as polydisks, however, ECH capacities do not always yield sharp obstructions. For instance, while ECH capacities imply that if the polydisk $P(2, 1)$ symplectically embeds into $B(c)$ then $c \ge 2$ \cite{H2}, Hind and Lisi \cite{HL} were able to improve the bound on $c$ to $3$.  This bound on $c$ is optimal, in the sense that 2 is the largest value of $a$ such that $P(a,1)$ symplectically embeds into $B(c)$ if and only if $c \geq a+1$.  A more detailed comparison between the results of this paper and the obstructions from ECH capacities are given in Section \ref{ECHcap}.
\end{remark}

The method of \emph{symplectic folding} is used to construct countless examples of nontrivial embeddings $X_1\xhookrightarrow{s} X_2$, including many cases of the form $X_1=P(a,b)$. The following remark elucidates that \emph{long} polydiscs ($a>2$) fold nicely, and in a sense, the longer the polydisc, the better it folds. The results in this paper prohibit the possibility of embedding \emph{short} polydiscs ($a<2$) into ellipsoids, and hence address embedding questions disjoint from those answered by symplectic folding. 

\begin{remark}
If $a>2$, then  folding $P(a,1)$ once provides a symplectic embedding into the ball $B(2+ a/2+\varepsilon)$ (\cite[\S 4.3.2]{Sch}). Multiple symplectic foldings may be used to produce better embeddings into balls when $a>6$; a multiple folding of $P(7,1)$ provides an embedding into $B(21/4+\varepsilon)$ (\cite[Prop.~ 4.10]{Sch}), which is stronger than the embedding into $B(11/2+\varepsilon)$, provided by folding $P(7,1)$ only once. Conversely, if $1\leq a\leq 2$, multiple symplectic foldings can not provide a stronger embedding result than the one provided by the inclusion of $P(a,1)$ into $B(1+a)$ (\cite[\S 4.3]{Sch}).
\end{remark}

Hutchings \cite{H} subsequently studied the information coming from embedded contact homology in a more refined way to provide a combinatorial means for finding better obstructions to embeddings of convex toric domains.  In particular, Hutchings reproved the result of Hind-Lisi and extended it to obstruct symplectic embeddings of other polydisks into balls. Subsequent work by \cite{CN}, extended the ``sharp" range of $a$ to $2.4 < a \leq \frac{\sqrt{7}-1}{\sqrt{7}-2} $ :

\begin{theorem}\emph{(\cite[Thm.~1.3]{H}, \cite[Thm.~1.4]{CN})}
Let $2 \leq a \leq \frac{\sqrt{7}-1}{\sqrt{7}-2} \approx 2.54858 $. If $P(a,1)$ symplectically embeds into $B(c)$ then
$c \geq 2+a/2.$

\end{theorem}

We now turn our attention to when the target is an ellipsoid, rather than a ball. Our first result is the following extension of \cite[Thm. 1.5]{H}, concerning symplectic embeddings of polydisks $P(a, 1)$ into the ellipsoid $E(bc, c)$.  Hutchings proved that if $1 \le a \le 2$ and $b$ is a positive integer then $P(a, 1)$ symplectically embeds into $E(bc, c)$ if and only if $a + b \le bc$.   We note that $a + b \le bc$ holds precisely when $P(a, 1)\subset E(bc, c)\subset\C^2$.  We extend this result to allow $b$ to be a half integer, under additional mild assumptions, as follows:

	\begin{theorem}\label{thm: NY1}
		Let $d_0 \ge 3$ be a prime number. Let $1 \le a \le (2d_0 -1)/d_0$, $c > 0$ and $b = p/2$ for some odd integer $p \ge 4d_0+1$. Then $P(a,1)$ symplectically embeds into $E(bc, c)$ if and only if $a+b \le bc$.
	\end{theorem}

\begin{remark}\label{rmk:hypothesis}
	Our hypothesis imposes restrictions on the values of $a$ that relate to the restriction on $b=p/2$. As $p$ increases, Theorem \ref{thm: NY1} works for larger $a$ values, approaching $a = 2$. When taking $d_0 = 3$, the result is for $1 \le a \le 5/3$ and odd integers $p \ge 13$. 
\end{remark}

	For smaller values of $p$, we are able to extract refined information from the combinatorics driving the proof of Theorem \ref{thm: NY1}. We obtain:
	
	\begin{theorem}\label{thm: NY2}
		Let $1 \le a \le 4/3$, $c > 0$ and $b = p/2$ for some odd integer $p > 2$. Then $P(a,1)$ symplectically embeds into $E(bc, c)$ if and only if $a+b \le bc$.
	\end{theorem}
	
	\begin{theorem}\label{thm: NY3}
		Let $1 \le a \le 3/2$, $c > 0$ and $b = p/2$ for some odd integer $p \ge 7$. Then $P(a,1)$ symplectically embeds into $E(bc, c)$ if and only if $a+b \le bc$.
	\end{theorem}

\begin{remark}
 Theorem \ref{thm: NY1} does not provide any information for embeddings featuring $p < 13$. The following examples demonstrate some of the different ground covered between Theorems \ref{thm: NY1}, \ref{thm: NY2}, and \ref{thm: NY3}:
\begin{itemize}
    \item Given a symplectic embedding $P(4/3,1)\hookrightarrow E(3c/2,c)$, Theorem \ref{thm: NY2} guarantees that $c\geq 17/9$ whereas Theorems \ref{thm: NY1} and \ref{thm: NY3} provide no restriction on $c$.
    \item Given a symplectic embedding $P(3/2,1)\hookrightarrow E(7c/2,c)$, Theorem \ref{thm: NY3} guarantees that $c\geq 10/7$ whereas Theorems \ref{thm: NY1} and \ref{thm: NY2} provide no restriction on $c$.
      \item Given a symplectic embedding $P(5/3,1)\hookrightarrow E(13c/2,c)$, Theorem \ref{thm: NY1} guarantees that $c\geq 49/13$ whereas Theorems \ref{thm: NY2} and \ref{thm: NY3} provide no restriction on $c$.
\end{itemize}
We note that the application of Theorem \ref{thm: NY1} to a choice of $p\geq 13$ provides a stronger statement of the application of either of Theorems \ref{thm: NY2} or \ref{thm: NY3} to the same odd integer $p$.
\end{remark}

	\begin{remark}
For larger $a$ values (but restricted $c$ values), Hind-Zhang have proven an analogous result \cite[Thm. 1.5(2)]{HZ} of embeddings of polydisks into half integer ellipsoids;  namely for $a \geq 2$, $b \in \N_{\geq 2}$, and $1\le c \le 2$, there is a symplectic embedding of $P(a,1)$ into $E(bc,c)$ if and only if $a+b \leq bc$.  Hind-Zhang's upper bound on $c$ is necessary to exclude folding and their sharp obstruction is obtained via the (reduced) shape invariant, which encodes the possible area classes of embedded Lagrangian tori in star-shaped domains of $\C^2$.
\end{remark}

\begin{remark}	
{There does not seem to be reason to expect} that the conclusions of \cite[Thm. 1.5]{H} for integral ellipsoids and our extension to half integer ellipsoids, Theorem \ref{thm: NY1}, would fail to hold for general ellipsoids.  However, at present it is unclear how to utilize the Huchings criterion to obtain the expected obstructions.  In Section \ref{sec:moreq}, we provide further insight on the combinatorial subtleties encountered when trying to generalize the methods of our proof for arbitrary positive rational $b = p/q$, where $p, q$ are relatively prime integers, and {$1 \le a < 2$}; see also Remark \ref{criterion-limitations}. 
\end{remark}

Before reviewing Huchings' combinatorial criterion, which is used to prove our main results, we provide a comparison with obstructions from ECH capacities.

\subsection{{Comparison with obstructions from ECH capacities }}\label{ECHcap}

We present several examples of the gap between the obstructions to polydisk embeddings from ECH capacities and the stronger obstructions from  \cite[Thm.~1.3]{H}, Theorems \ref{thm: NY2} and \ref{thm: NY3}; see also Remark \ref{rmk:echcapP}. We would like to thank the anonymous referee for suggesting the methods used in Examples \ref{ex:a3/2b2} and \ref{ex:a1b3/2}.

The ECH capacities of a toric domain $X$ (with the standard symplectic form) are an increasing sequence of numbers derived from the ECH chain complex of $\partial X$ (with the standard contact form). As noted earlier in (\ref{eqn:capineq}), if $X$ symplectically embeds into $X'$, then each ECH capacity of $X$ is bounded from above by the corresponding ECH capacity of $X'$ (this is a version of \cite[Thm.~1.1]{H2} for embeddings between possibly closed manifolds). Generalizing the work of McDuff \cite{dusaellipsoid2}, Cristofaro-Gardiner proved that in some cases (\ref{eqn:capineq}) actually implies symplectic embedding:
\begin{theorem}[{\cite[Thm. 1.2]{CG}}]\label{thm:CG}
    If $(X,\omega)$ is a concave toric domain and $(X',\omega')$ is a convex\footnote{note that the terminology ``convex" in \cite{CG} is slightly broader than ours} toric domain, there is a symplectic embedding $\operatorname{int}(X)\hookrightarrow\operatorname{int}(X')$ if and only if
    \[
    c_k(\operatorname{int}(X))\leq c_k(\operatorname{int}(X')) \mbox{ for all } k
    \]
\end{theorem}

\begin{remark} Theorem \ref{thm:CG} concerns embeddings of only the interiors of toric domains, while the results in \cite{H} and our results refer to embeddings of the entire domain, boundary included. This is because in our setting, the embeddings which do exist are simply inclusion, while many of the embeddings covered by Theorem \ref{thm:CG} are more exotic. Note also that $c_k(\operatorname{int}(X))=c_k(X)$ for toric domains by the definition of the ECH capacities of a non-closed region: \cite[Def.~4.9]{H2}.
\end{remark}

Theorem \ref{thm:CG} underlies much of the recent work on the role of ECH in symplectic embeddings, including but not limited to \cite{CGHMP, BHMMMPW}. However, the embeddings we obstruct in this paper are of the form
\[
P(a,1)\hookrightarrow E(bc,c)
\]
for various ranges of $a,b$, and $c$, and therefore Theorem \ref{thm:CG} does not apply. In this case, an obstruction to an embedding is a lower bound on $c$ in terms of $a$ and $b$. We will show several examples in which the lower bound on $c$ obtained via  \cite[Thm. 1.3]{H}, Theorems \ref{thm: NY2} and \ref{thm: NY3} is larger than any lower bound, which is possible to obtain via ECH capacities.

First, we explain how to obtain a lower bound on $c$ from ECH capacities, which is a special case of a more general fact: see \cite[(2.9)]{CGHMP} and \cite[(2.5)]{H2}.
\begin{lemma}\label{lem:clb}
    If $P(a,1)$ symplectically embeds into $E(bc,c)$, then
    \[
    c\geq\sup_k\frac{c_k(P(a,1))}{c_k(E(b,1))}.
    \]
\end{lemma}
\begin{proof}
If $P(a,1)$ symplectically embeds into $E(bc,c)$, then $c_k(P(a,1))\leq c_k(E(bc,c))$ for all $k$. \\ \textbf{Claim:} $c_k(E(bc,c))=c\cdot c_k(E(b,1))$. \\ Assuming the claim, we then know
\[
c_k(P(a,1))\leq c\cdot c_k(E(b,1)) \;\forall k,
\]
which implies our conclusion $c\geq\sup_k\frac{c_k(P(a,1))}{c_k(E(b,1))}$, by dividing both sides by $c_k(E(b,1))$ and taking the supremum over $k$.

\textbf{Proof of claim:} it suffices to show that the Reeb orbits on $\partial E(bc,c)$ are in 1-1 correspondence with the Reeb orbits on $\partial E(b,1)$ and that this correspondence divides Reeb orbit length by $c$. Using the standard contact form $\lambda=\frac{1}{2}\sum_{i=1}^2r_i^2\,d\theta_i$ in polar coordinates on $\C^2$, it is straightforward to compute that the Reeb vector field\footnote{Given a cooriented contact manifold $(M,\mbox{ker}\lambda)$, the Reeb vector field $R$ is uniquely determined by the equations $\lambda(R)=1$ and $d\lambda(R, \cdot) =0$.} on $\partial E(bc,c)$ is
\[
\frac{2\pi}{bc}\frac{\partial}{\partial\theta_1}+\frac{2\pi}{c}\frac{\partial}{\partial\theta_2}=\frac{1}{c}\left(\frac{2\pi}{b}\frac{\partial}{\partial\theta_1}+2\pi\frac{\partial}{\partial\theta_2}\right),
\]
while the Reeb vector field on $\partial E(b,1)$ is
\[
\frac{2\pi}{b}\frac{\partial}{\partial\theta_1}+2\pi\frac{\partial}{\partial\theta_2}.
\]
Therefore on each torus of constant $(r_1,r_2)$ the Reeb vector fields on $\partial E(bc,c)$ and $\partial E(b,1)$ sweep out the same distribution, while the Reeb flow on $\partial E(bc,c)$ takes $c$ times longer to close up. Finally, it remains to show that the tori of constant $(r_1,r_2)$ of each ellipsoid are in 1-1 correspondence. We do this by identifying the torus in $\partial E(b,1)$ above $(r_1,r_2)\in\R^2$  with the torus in $\partial E(bc,c)$ above $(cr_1,cr_2)\in\R^2$.
\end{proof}

Note that our claim in the proof of Lemma 1.15 is an elementary case of the conformality property of ECH capacities, \cite[Thm. 1.3]{Hu2}, highlighting the direct correspondence of Reeb orbits in this case (and avoiding the more abstract machinery necessary to prove conformality in general).



In Examples \ref{ex:a3/2b2}-\ref{ex:a3/2b3/2} we discuss three cases where the lower bound on $c$ from (\ref{lem:clb}) is smaller than the lower bound $c\geq a/b+1$ from \cite[Thm.~1.3]{H}, Theorems \ref{thm: NY2} and \ref{thm: NY3}. Before discussing the first example, we introduce some results and terminology, which allows us to compute ECH capacities and thus analyze the supremum in Lemma \ref{lem:clb}.

First we have the formula for capacities of disjoint unions.
\begin{proposition}[{\cite[Prop.~1.5]{H2}}] If $X_0$ and $X_1$ are toric domains, then
\[
c_k\left(X_0\coprod X_1\right)=\max_{i+j=k}c_i(X_0)+c_j(X_1).
\]
\end{proposition}

In Example \ref{ex:a3/2b2} we will compute of the ECH capacities of a convex toric domain from its ``negative weight expansion," a method due to Cristofaro-Gardiner and Choi. The \textit{negative weight expansion} (defined in \cite[\S2.2]{CG}) of a convex toric domain $X_\Omega$ is the sequence $(b;b_1,b_2,\dots)$ obtained inductively as follows. Define a ``$b$-triangle" to be any triangle affine equivalent (e.g. up to $GL(2,\Z)$ transformations and translations) to the triangle with vertices $(0,0), (b,0)$, and $(0,b)$. The first entry $b$ in the negative weight expansion of $\Omega$ is the smallest $b$ so that $\Omega$ fits inside a $b$-triangle. If $\Omega$ is itself a $b$-triangle, then we are done. If not, the second entry $b_1$ is the largest number so that a $b_1$-triangle fits inside the complement of $\Omega$ in the $b$-triangle and sharing a corner with the $b$-triangle. If $\Omega$ and this $b_1$-triangle fully fill the $b$-triangle, then we are done; else proceed inductively.

We can now compute the ECH capacities of convex toric domains, such as polydisks and ellipsoids:
\begin{theorem}[{\cite[Thm. A.1]{CG}}]\label{thm:CGKC} Let $X_\Omega$ be a convex toric domain with negative weight expansion $(b;b_1,b_2)$. Then
\[
c_k(X_\Omega)=\inf_{\ell\geq0}c_{k+\ell}(B(b))-c_\ell\left(B(b_1)\coprod B(b_2)\right).
\]
\end{theorem}

\begin{example}\label{ex:a3/2b2} This example concerns the embedding $P(3/2,1)\hookrightarrow E(7/2,7/4)$, which is the $a=\frac{3}{2}, b=2$ case of \cite[Thm. 1.5]{H}. The conclusion is that $P(3/2,1)$ cannot symplectically embed into $E(2c,c)$ for any $c<7/4$. It is not possible to obstruct such embeddings using ECH capacities alone. We show that $\sup_k\frac{c_k(P(3/2,1))}{c_k(E(2,1))}=5/4$, which is realized by the ratio of the third ECH capacities:\footnote{Here we used the ECH capacities of a ball (\cite[Cor. 1.3]{H2}) and Thm. \ref{thm:CGKC} to obtain this computation.}
\[
\frac{c_3\left(P\left(\frac{3}{2},1\right)\right)}{c_3(E(2,1))}=\frac{5/2}{2}=\frac{5}{4}.
\]

By reversing the proof of Lemma (\ref{lem:clb}), what we need to show is that
\[
\begin{array}{rclc}
\displaystyle \sup_k\frac{c_k(P(3/2,1))}{c_k(E(2,1))}&\leq&\displaystyle \frac{5}{4}&\Leftrightarrow \\
&&&\\
c_k\left(P\left(\frac{3}{2},1\right)\right)&\leq& c_k\left(E\left(\frac{5}{2},\frac{5}{4}\right)\right)\;\forall k&\Leftrightarrow\\
 c_k\left(P\left(\frac{6}{5},\frac{4}{5}\right)\right)&\leq& c_k(E(2,1)) \;\forall k,&\\
\end{array}
\]
by scaling both $|z_i|^2$ coordinates by $\frac{4}{5}$. We perform this scaling so that the negative weight expansions both have $b=2$. That the negative weight expansion of $P(6/5,4/5)$ is $(2;6/5,4/5)$ is immediate. To compute the negative weight expansion of $E(2,1)$, note that it must have $b=2$, leaving an obvious $1$-triangle in the complement of $\Omega$ with corners at $(0,1), (1,1)$, and $(0,2)$, and another triangle with corners $(1,1), (0,1)$, and $(2,0)$. After translating by $(-2,0)$, the shear transformation $\begin{pmatrix}1&2\\-1&-1\end{pmatrix}$ turns the second triangle into another standard $1$-triangle. Therefore the negative weight expansion of $E(2,1)$ is $(2;1,1)$.

Using Theorem \ref{thm:CGKC}, our goal is now to show that for all $k$,
\[
\inf_{\ell\geq0}c_{k+\ell}(B(2))-c_\ell\left(B\left(\frac{6}{5}\right) \coprod B\left(\frac{4}{5}\right)\right)\leq \inf_{\ell\geq0}c_{k+\ell}(B(2))-c_\ell\left(B(1)\coprod B(1)\right),
\]
which, because the $b$ terms in the negative weight expansions of $P(6/5,4/5)$ and $E(2,1)$ are equal, follows from showing that for all $\ell$,
\begin{align}
    c_\ell\left(B\left(\frac{6}{5}\right)\coprod B\left(\frac{4}{5}\right)\right)&\geq c_\ell\left(B(1)\coprod B(1)\right)\nonumber
    \\\max_{i+j=\ell}c_i\left(B\left(\frac{6}{5}\right)\right)+c_j\left(B\left(\frac{4}{5}\right)\right)&\geq\max_{m+n=\ell}c_m(B(1))+c_n(B(1)).\label{eqn:ijmn}
\end{align}
By the fact that $c_k(B(a))=ac_k(B(1))$ by \cite[Cor. 1.3]{H2}, our goal (\ref{eqn:ijmn}) would follow if for all $m,n$ with $m+n=\ell$, we could identify some $i,j$ with $i+j=\ell$ for which
\[
\frac{6}{5}c_i(B(1))+\frac{4}{5}c_j(B(1))\geq c_m(B(1))+c_n(B(1)).
\]
If $m\geq n$ then $i=m, j=n$ is such a pair; if $m<n$ then set $i=n, j=m$.
\end{example}

\begin{example}\label{ex:a1b3/2} This example concerns the embedding $P(1,1)\hookrightarrow E(5/2,5/3)$. Theorem \ref{thm: NY2} with $a=1$ and $b=\frac{3}{2}$ proves that if $P(1,1)$ symplectically embeds into $E(3c/2,c)$ then $c\geq\frac{5}{3}$. We claim that the best lower bound, which is possible to obtain from the ratio of ECH capacities is $c\geq\frac{4}{3}$, is the ratio of the second ECH capacities of $P(a,1)$ and $E(b,1)$:
\[
\frac{c_2(P(1,1))}{c_2\left(E\left(\frac{3}{2},1\right)\right)}=\frac{2}{3/2}=\frac{4}{3}.
\]
First, note that $P(1,1)$ has the same ECH capacities as $E(2,1)$, because they have the same negative weight expansion $(2;1,1)$. Next, notice that $E(2,1)\subset\frac{4}{3}E(3/2,1)=E(2,4/3)$. Therefore, there can be no larger lower bound on $c$ derived from the ECH capacities of $P(1,1)$, because we would then obtain an obstruction via the equal ECH capacities of $E(2,1)$ to the embedding $E(2,1)\hookrightarrow E(2,4/3)$.
\end{example}

For our final example, we invoke a simplified version of the Weyl law for ECH capacities, which relates their asymptotic behavior to the volume of $X$ computed with $\frac{1}{2}\omega\wedge\omega$.
\begin{theorem}[{\cite[Thm. 1.1]{CGHR}}]\label{thm:CGHR}
        For toric domains $X$,
        \[
        \lim_k\frac{c_k(X)^2}{k}=4\operatorname{vol}(X).
        \]
\end{theorem}
In the case of toric domains, $\operatorname{vol}(X_\Omega)$ is proportional to the area of $\Omega$. This means we expect the obstructions to embeddings of $X_\Omega$ from ECH capacities to get worse as $k$ gets large, since they limit to $k$ times the area of $\Omega$ and do not take its shape into account. One can think of this as a long-term tendency towards the phenomenon we saw in Example \ref{ex:a1b3/2}, where the ECH capacities of $P(1,1)$ and $E(2,1)$ (whose domains $\Omega$ in $\R^2$ are both of area one) were equal for all $k$.

\begin{example}\label{ex:a3/2b3/2} We consider here the embedding $P(3/2,1)\hookrightarrow E(5,10/7)$. Theorem \ref{thm: NY3} with $a=3/2, b=7/2$ proves that if $P(3/2,1)$ symplectically embeds into $E(7c/2,c)$ then $c\geq10/7$. We expect that ECH capacities only provide the lower bound $c\geq1$, which is achieved by the ratio of the first ECH capacities:
\[
\frac{c_1\left(P\left(\frac{3}{2},1\right)\right)}{c_1\left(E\left(\frac{7}{2},1\right)\right)}=\frac{1}{1}=1.
\]
To obtain the lower bound $c\geq1$, we computed the maximum of the ratios of the first 25,000 ECH capacities of the polydisk $P(3/2,1)$ and ellipsoid $E(7/2,1)$, using the computer program Mathematica. Our methods are analogous to those of \cite[\S5]{BHMMMPW}, adapted for the polydisk. It is highly likely that we have in fact found the supremum, because by Theorem \ref{thm:CGHR},
\[
\lim_{k\to\infty}\frac{c_k(X)}{c_k(Y)}=\sqrt{\frac{\text{vol}(X)}{\text{vol}(Y)}},
\]
and therefore it's unlikely that a sequence limiting to 
\[
\sqrt{\operatorname{vol}(P(3/2,1))/\operatorname{vol}(E(7/2,1))}=\sqrt{6/7}\approx0.85714
\] with its first 25,000 terms less than or equal to $1$ will include terms as large as $10/7\approx1.42857$. It may be possible to use the methods of Example \ref{ex:a3/2b2} or results of \cite{W} to prove this, but we will not attempt to do so here.
\end{example}

	\subsection{Review of convex generators and the Hutchings criterion}\label{section: review}

We now review the principal combinatorial objects involved in stating the Hutchings criterion \cite[Thm. 1.19]{H}, which is necessary for one convex toric domain to be symplectically embedded into another convex toric domain.  In Section \ref{bigsection2}, we will use this apparatus to prove Theorems \ref{thm: NY1}, \ref{thm: NY2}, and \ref{thm: NY3}.  

	\begin{definition}\label{defn: convexpath}
	A \emph{convex path} (in the first quadrant) is a path $\Lambda$ in the plane such that:
	\begin{itemize}
	    \item The endpoints of $\Lambda$ are $(0, y(\Lambda))$ and $(x(\Lambda), 0)$ where $x(\Lambda)$ and $y(\Lambda)$ are non-negative real numbers.
	    \item $\Lambda$ is the graph of a piecewise linear concave function $f: [0, x(\Lambda)] \to [0, y(\Lambda)]$ with $f'(0) \le 0$, possibly together with a vertical line segment at the right.
	\end{itemize}
	{The \textit{vertices} of $\Lambda$ are the points at which its slope changes, including its endpoints. Its \textit{edges} are the line segments between vertices.} $\Lambda$ is called a \emph{convex integral path} if, in addition, 
	\begin{itemize}
	    \item $x(\Lambda)$ and $y(\Lambda)$ are integers.
	    \item The vertices of $\Lambda$ are lattice points.
	\end{itemize}
	\end{definition}
	
	
	\begin{definition}
		A \emph{convex generator} is a convex integral path $\Lambda$ such that:
		\begin{itemize}
			\item Each edge of $\Lambda$ is labeled `$e$' or `$h$'.
			\item Horizontal and vertical edges can only be labeled `$e$'.
		\end{itemize}
	\end{definition}
	
	\begin{remark}
	In our proofs we will use the following notation for convex generators.   If $a$ and $b$ are relatively prime nonnegative integers, and if $m$ is a positive integer, then:
\begin{itemize}
\item We use $e_{a,b}^m$ to denote an edge whose displacement vector is $(ma, -mb)$, labeled ``$e$"; 
\item  We use $h_{a,b}$ to denote an edge with displacement vector $(a, -b)$, labeled ``$h$"; 
\item Finally, if $m > 1$ then $e_{a,b}^{m-1}h_{a,b}$ denotes an edge with displacement vector $(ma, -mb)$, labeled ``$h$".
	 \end{itemize}
  A convex generator can thus be represented by a commutative formal product of the symbols $e_{a,b}$ and $h_{a,b}$, where no factor $h_{a,b}$ may be repeated, and the symbols $h_{1, 0}$ and $h_{0,1}$ may not be used.  
	\end{remark}

\begin{definition}
Let $\Lambda_1$ and $\Lambda_2$ be convex generators. We say that they ``\emph{have no elliptic orbit in common}" if the formal products corresponding to $\Lambda_1$ and $\Lambda_2$ share no common factor $e_{a,b}$. Similarly, we say that $\Lambda_1$ and $\Lambda_2$ ``\emph{have no hyperbolic orbit in common}" if the formal products representing $\Lambda_1$ and $\Lambda_2$ share no common factor $h_{a,b}$. If $\Lambda_1$ and $\Lambda_2$ have no hyperbolic orbit in common, we define their ``\emph{product}" $\Lambda_1\Lambda_2$ by concatenating the formal products corresponding to $\Lambda_1$ and $\Lambda_2$. This product operation is associative. 
\end{definition}
	
	\begin{definition}
	The quantity $m(\Lambda)$ is the \emph{total multiplicity} of all the edges of  $\Lambda$, i.e.\ the total exponent of all factors of $e_{a,b}$ and $h_{a,b}$ in the formal product for $\Lambda$. Note that $m(\Lambda)$ is equal to one less than the number of lattice points on the path $\Lambda$.
\end{definition}

Remarkably, as explained in \cite[\S 6]{H}, the boundary of any convex toric domain can be perturbed so that for its induced contact form, and up to large {symplectic} action, the ECH generators correspond to these convex generators.  As a result, the ECH index may be computed combinatorially in terms of lattice point enumeration. 

	\begin{definition} \label{defn: index}
		If $\Lambda$ {is} a convex generator, then its \emph{ECH index} is defined to be 
		\[
		I(\Lambda) = 2(L(\Lambda) - 1) - h(\Lambda),
		\]
		where $L(\Lambda)$ denotes the number of lattice points interior to and on the boundary of the region enclosed by $\Lambda$ and the $x, y$-axes, and $h(\Lambda)$ denotes the number of edges of $\Lambda$ that are labeled ``$h$".
	\end{definition}
Section \ref{section: counting}	provides computations of the ECH index of several convex generators of interest in symplectic embedding problems.  Section \ref{section: nature} provides relations the ECH index must satisfy for a convex generator with fixed endpoints.  These results, in combination with additional bounds coming from action and $J$-holomorphic {curve} genus, enable us to prove our main results.


	\begin{definition}\label{def:tangency}
		If $\Lambda$ is a convex generator and $X_\Omega$ is a convex toric domain, define the \emph{symplectic action} of $\Lambda$ with respect to $X_\Omega$ by 
		\[
		A_{X_\Omega}(\Lambda) = \sum_{\nu \in \text{Edges}(\Lambda)} \vec{\nu}\times p_{\Omega, \nu}.
		\]
 Here, for any edge $\nu$ of $\Lambda$, $\vec{\nu}$ denotes the displacement
  vector of $\nu$,  and $p_{\Omega, \nu}$ denotes any point on the line $\ell$
  parallel to $\vec{\nu}$ and tangent to $\partial\Omega$.  Tangency means that
  $\ell$ touches $\partial\Omega$ and that $\Omega$ lies entirely in one
  closed half plane bounded by $\ell$. Moreover, `$\times$' denotes the
   the determinant of the matrix whose columns are given by the
  two vectors.
	\end{definition}
	

Next, we compute the symplectic action of our favorite toric domains. 	
\begin{example}\label{eg: action}
If $X_\Omega$ is the polydisk $P(a, b)$, then 
\[
A_{P(a,b)}(\Lambda) = bx(\Lambda) + ay(\Lambda).
\]
If $X_\Omega$ is the ellipsoid $E(a, b)$, then
\[
A_{E(a,b)}(\Lambda) = c,
\]
where the line $bx + ay = c$ is tangent to $\Lambda$.
\end{example}

The following definition is essential for combinatorially computing ECH capacities:	
	\begin{definition}\label{def: minimalgen}
		If $X_\Omega$ is a convex toric domain, then a convex generator $\Lambda$ with $I(\Lambda) = 2k$ is said to be \emph{minimal} for $X_\Omega$ if:
		\begin{itemize}
			\item All edges of $\Lambda$ are labeled ``$e$".
			\item $\Lambda$ uniquely minimizes $A_\Omega$ among convex generators with $I = 2k$ and all edges labeled ``$e$".
		\end{itemize}

	\end{definition}

The symplectic action of minimal generators is related to ECH capacities as follows.
\begin{remark} 
If $I(\Lambda)=2k$ and $\Lambda$ is minimal for $X_\Omega$, then $A_\Omega(\Lambda) = c_k(X_\Omega)$, by \cite[Prop.~5.6]{H}.
\end{remark}

{In Section \ref{subsection:classification} we prove a converse to \cite[Lem. 2.1(a)]{H}}, enabling the following characterization of minimal generators associated to an ellipsoid $E(bc,c)$. We define a convex integral path $\Lambda$ to be \emph{maximal} under a convex path $\Gamma$ in $\R^2$ if $\Lambda$ exactly encloses all lattice points in the first quadrant enclosed by $\Gamma$, including those on $\Gamma$.

\begin{proposition}\label{prop: classification-intro}
Let $b\ge 1$, $c>0$. Let $\Lambda$ be any purely elliptic convex generator, and $\eta$ the line of slope $-1/b$ tangent to $\Lambda$. Then $\Lambda$ is minimal for the ellipsoid $E(bc,c)$ if and only if $\Lambda$ is maximal under $\eta$.
\end{proposition}

\begin{remark}
In the special case of $b=1$, that is, $E(bc,c)$ is the ball $B(c)$, Proposition \ref{prop: classification-intro} gives an alternative proof of \cite[Lem.~A.3]{CN}, which classifies all minimal generators for the ball $B(c)$.
\end{remark}

For half integer ellipsoids we can explicitly characterize the minimal generators as follows.

\begin{example}\label{eg: minimal} Let $c > 0$, let $d_0$ be a positive integer, and let $p$ be any positive integer. Then  by \cite[Lem.~2.1]{H} the convex generator $e_{p, 2}^{d_0}$ is minimal for the ellipsoid $E(pc/2, c)$. When $p/2 \geq 1$, then  by Proposition \ref{prop: p=2classification}  all minimal convex generators of $E(pc/2, c)$ are of the form 
$e_{1, 0}^k e_{\frac{p+1}{2}, 1}^{m_1} e_{p, 2}^d e_{\frac{p-1}{2}, 1}^{m_2}$, where $m_i\in\{0,1\}$, $d\geq0$, and $0 \le k < \frac{p+1}{2}$ if $m_1 = 0$. 

\end{example}

	
Our final definition will be key to understanding when one convex toric domain can be symplectically embedded into another convex toric domain.
\begin{definition}\label{defn: Jcurve}
Let $\Lambda, \Lambda'$ be convex generators such that all edges of $\Lambda'$ are labeled {``$e$,"} and let $X_\Omega$, $X_{\Omega'}$ be convex toric domains. We write $\Lambda \le_{X_\Omega, X_{\Omega'}} \Lambda'$ if the following three conditions hold:
	\begin{enumerate}[(i)]
		\item Index requirement: $I(\Lambda) = I(\Lambda')$.
		\item Action inequality: $A_{\Omega}(\Lambda) \le A_{\Omega'}(\Lambda')$.
		\item $J$-holomorphic curve genus inequality: $x(\Lambda) + y(\Lambda) - \frac{h(\Lambda)}{2} \ge x(\Lambda') + y(\Lambda') + m(\Lambda') - 1$.
	\end{enumerate}
	\end{definition}
In particular, if $X_\Omega$  symplectically embeds into $X_{\Omega'}$, then the resulting cobordism between their (perturbed) boundaries implies that $\Lambda \leq_{X_\Omega,X_{\Omega'}} \Lambda'$ is a necessary condition for the existence of an embedded irreducible holomorphic curve with ECH index zero between the ECH generators corresponding to $\Lambda$ and $\Lambda'$.  The {name of the} third inequality arises from the fact that every $J$-holomorphic curve must have nonnegative genus, see \cite[Prop.~3.2]{H}, and {proving that a $J$-holomorphic curve must exist in cobordisms resulting from embeddings of convex toric domains} is ultimately {what} allowed Hutchings to go ``beyond" ECH capacities in his criterion.

In Section \ref{section: restriction}, we utilize the $J$-holomorphic curve genus inequality with the action inequality to find further restrictions on the endpoints of convex generators, which when used in combination with the ECH index requirement, ultimately allow us to obtain obstructions through Hutchings' combinatorial criterion.

	

	We now have all the ingredients to state the Hutchings criterion: 
	
	\begin{theorem}{\em \cite[Thm.~1.19]{H}}\label{thm: hu criterion}
		Let $X_\Omega$ and $X_{\Omega'}$ be convex toric domains. Suppose there exists a symplectic embedding $X_\Omega \to X_{\Omega'}$. Let $\Lambda'$ be a convex generator which is minimal for $X_{\Omega'}$. Then there exists a convex generator $\Lambda$ with $I(\Lambda) = I(\Lambda')$, a nonnegative integer $n$, and product decompositions $\Lambda = \Lambda_1 \cdots \Lambda_n$ and $\Lambda' = \Lambda'_1 \cdots \Lambda'_n$, such that
		\begin{enumerate}[\em (i)]
			\item $\Lambda_i \le_{\Omega, \Omega'} \Lambda'_i$ for each $i = 1, \dots, n$.
			\item Given $i, j \in \{1, \dots, n\}$, if $\Lambda_i \ne \Lambda_j$ or $\Lambda'_i \ne \Lambda'_j$, then $\Lambda_i$ and $\Lambda_j$ have no elliptic orbit in common.
			\item If $S$ is any subset of $\{1, \dots, n\}$, then $I\left(\prod_{i\in S}\Lambda_i\right) = I\left(\prod_{i\in S}\Lambda'_i\right)$.
		\end{enumerate}
	\end{theorem}
	In practice, Theorem \ref{thm: hu criterion} is used in a negative way to provide obstructions to the symplectic embedding in question.  

The proofs of our main results begin by assuming the existence of a nontrivial embedding and checking the Hutchings criterion against minimal generators of the form $\Lambda' = e_{p,2}^{d_0}$ (which are minimal for $E(pc/2, c)$, by Example \ref{eg: minimal}). As a result, we obtain a convex generator $\Lambda$ and the corresponding product decompositions of $\Lambda$ and $e_{p,2}^{d_0}$ into $n$ factors. Our plan is then to eliminate all possible factorizations of $\Lambda$ through the combinatorial conditions that Theorem \ref{thm: hu criterion} mandates and thus achieve a contradiction. In the process of elimination, we start by restricting possibilities of $\Lambda$ using the first condition of the Hutchings criterion, which unfolds into the three requirements of Definition \ref{defn: Jcurve}. The remaining possibilities for $\Lambda$ and their factorizations will then be eliminated using the action inequaity (ii) and the $J$-holomorphic genus inequality (iii) in Definition \ref{defn: Jcurve}.

\begin{remark}\label{criterion-limitations}
 In Section \ref{bigsection3}, we provide abstract examples, Propositions \ref{prop: example1} - \ref{prop: example3}, to illustrate the limitations in using the Hutchings criterion when applied to a minimal generator of the form $e_{p,q}^{d_0}$ to extend Theorem \ref{thm: NY1}.  Propositions \ref{prop: example1} - \ref{prop: example2} provide two abstract examples, relevant to the symplectic embedding problem of $P(a, 1)$ into $E(pc/2, c)$ for any $a>(2d_0-1)/d_0$ or $p< 4d_0-1$. In Propositions \ref{prop: example3}, we generalize the half integer $b$ to $p/q$ for all $q > 3$ and consider the embedding problem of $P(a, 1)$ into $E(bc, c)$ when $1 < a \le (2d_0-1)/d_0$. For each of these embedding problems, we show that, when $2a+p-\varepsilon <pc<2a+p$ is satisfied for some $\varepsilon > 0$, there always exists a convex generator $\Lambda$ with factorizations that satisfy the three conditions of Theorem \ref{thm: hu criterion}, and thus no contradiction of any kind can be achieved when taking $\Lambda' = e_{p,2}^{d_0}$.
\end{remark}	

\begin{remark}\label{criterion-limitations2}
In Section \ref{subsection:classification}, we classify minimal convex generators associated to the ellipsoid, as summarized in Proposition \ref{prop: classification-intro} and Example \ref{eg: minimal}.  We had hoped we could use the latter generators, which represent more complicated lattice paths, in our application of the Hutchings criterion to extend the range our results in obstructing symplectic embeddings of polydisks into ellipsoids.  However, as explained in Remark \ref{rem:othergens}, further work is needed in this direction, as the combinatorial methods developed in this paper are inconclusive when applied to these more complex generators. 
\end{remark}


\begin{remark}	For the rest of this paper, we will use $d_0$ to denote the power of $e_{p,2}^{d_0}$ that we intend to apply the Hutchings criterion to. This is the ``top power" to consider. We will use $d \le d_0$ to denote the powers of factors of $e_{p,2}^{d_0}$. For brevity, we will use the symbol ``$\le$" in place of ``$\le_{P(a,1),E(bc,c)}$" between convex generators when $a, b$ and $c$ are specified without ambiguity. And we  will use the symbol ``$\xhookrightarrow{s}$" to mean ``symplectically embeds into."
\end{remark}

\noindent \textbf{Acknowledgements.}	We would like to thank Michael Hutchings for suggesting this project and for helpful discussions. We thank the referee for their careful reading and suggestions on how to clarify the results of this paper.  Our 2020 Virtual BeECH Group was supported by NSF grant DMS-1840723.  Additionally, LD was partially supported by NSF grants DMS-1745670, DMS-1840723, DMS-2104411; JN was partially supported by NSF grants DMS-1840723, DMS-2104411; and MW was partially supported by NSF grants DMS-1745670, DMS-2103245.


	\section{Embedding polydisks into ellipsoids}\label{bigsection2}


	The main goal of this section is to prove the nontrivial direction of Theorems \ref{thm: NY1}, \ref{thm: NY2}, and \ref{thm: NY3}.  In Section \ref{section: counting} we provide some formulae for the ECH index of several convex generators. Next, we characterize certain convex generators with fixed endpoints in Section \ref{section: nature}. In Section \ref{section: restriction}, we explore the restrictions on the convex generator $\Lambda$ for $P(a,1)$ satisfying $\Lambda \le_{P(a,1), E(pc/2, c)} e_{p,2}^d$, which is guaranteed to us by the Hutchings criterion, Theorem \ref{thm: hu criterion}, including the product decompositions of $\Lambda$ and $e_{p,2}^{d_0}$.
	
We further categorize the product decompositions of $\Lambda$ into three scenarios, dependent on distinct combinatorial features. We call them the ``trivial factorization," the ``general factorization," and the ``full factorization" for easy reference: 	
	\begin{enumerate}
	    \item The trivial factorization is the case where $n = 1$. By the first condition of the Hutchings criterion Theorem \ref{thm: hu criterion}, this implies $\Lambda \le \Lambda' = e_{p,2}^{d_0}.$ We will prove the non-existence of such a $\Lambda$ in Section \ref{section: triv}. 
	    \item The general factorization covers the case where $2 \le n \le d_0 - 1$. In Section \ref{section: gen} we prove Proposition \ref{prop: genfactor} that eliminates this possibility when combined with the primality of $d_0$.
	    \item The full factorization is the case where $n = d_0$. In this case, $\Lambda'_i = e_{p,2}$ for each $i \in \{1, \dots, n\}$. We will show in Section \ref{section: full} that this factorization cannot be achieved.
	\end{enumerate}
 In Section \ref{section: proof} we appeal to the elimination of these factorizations and the combinatorial restrictions 
to prove Theorems \ref{thm: NY1}, \ref{thm: NY2}, and \ref{thm: NY3}.  The latter two results are not direct corollaries of Theorem \ref{thm: NY1}, but rely on similar arguments, which we elucidate.


	\subsection{ECH index formulae}\label{section: counting}
	
	The following lemma provides a shortcut for computing the ECH index of several convex generators of  interest.
	
	\begin{lemma}\label{lemma: count}
		Suppose $k,m, d$ are nonnegative integers with $d \ge 1$.  Then     
		\begin{enumerate}[\em (i)]
		    \item $I(e_{1,0}^ke_{0,1}^m) = 2(km + k + m)$,
		    \item $I(e_{k,1}e_{0,1}^{m-1}) = 2(km+m)$,
			\item $I(e_{k,m}) = km+k+m+ {\rm gcd}(k,m)$,
			\item $I(e_{1,0}^{kd}e_{m,1}^d) = (2k+m)d^2 + (2k+m+2)d$.
\item	    If $\gcd(p,q) =1$ and $p,q \in \mathbb{N}_{>0}$ then $I(e_{p,q}^d) = pqd^2 +(p+q+1)d$.
		\end{enumerate}
	\end{lemma}

	\begin{proof}
	Suppose $k, m, d, p, q$ are as given.
	\begin{enumerate}[ (i)]
	    \item This follows from $L(e_{1,0}^ke_{0,1}^m) = (k+1)(m+1).$ 
	    \item This follows from $L(e_{k,1}e_{0,1}^{m-1}) = m(k+1)+1.$
	    \item The number of lattice points on the line segment connecting $(k, 0)$ and $(0, m)$, including the two endpoints, is precisely ${\rm gcd}(k,m) + 1$. 
	    Thus 
	    \[
	    L(e_{k,m}) = \frac{L(e_{1,0}^{k}e_{0,1}^{m}) + {\rm gcd}(k,m) + 1}{2}.
	    \]
	    Using (i) this gives $I(e_{k,m}) = km+k+m+ {\rm gcd}(k,m)$.
	    \item This follows from $L(e_{1,0}^{kd}e_{m,1}^d) = L(e_{1,0}^{kd}e_{0,1}^d) + L(e_{m,1}^d) - (d+1)$, as illustrated in Figure \ref{fig: lattice2}, and using (i) and (iii). 

	    \item This follows from (iii) by taking $k=pd$ and $m=qd$. Figure \ref{fig: lattice1} shows how $L(e_{p,q}^d)$ can be obtained using the same method as in (iii). In this case, the number of lattice points on the line segment is ${\rm gcd}(pd, qd) + 1 = d + 1$.

	    \end{enumerate}
	    	
	\begin{figure}[H]
	    \centering
	    \begin{subfigure}{0.45\textwidth}
		\def\svgwidth{\columnwidth}
		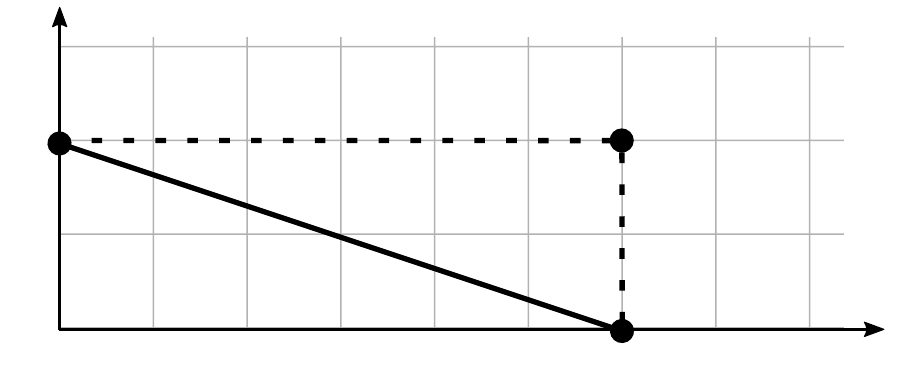
		\caption{The triangle encloses half of the lattice points enclosed by the rectangle, plus some on the slanted line segment.}\label{fig: lattice1}
	    \end{subfigure}
	    \hspace*{3mm}
	    \begin{subfigure}{0.45\textwidth}
		\def\svgwidth{\columnwidth}
		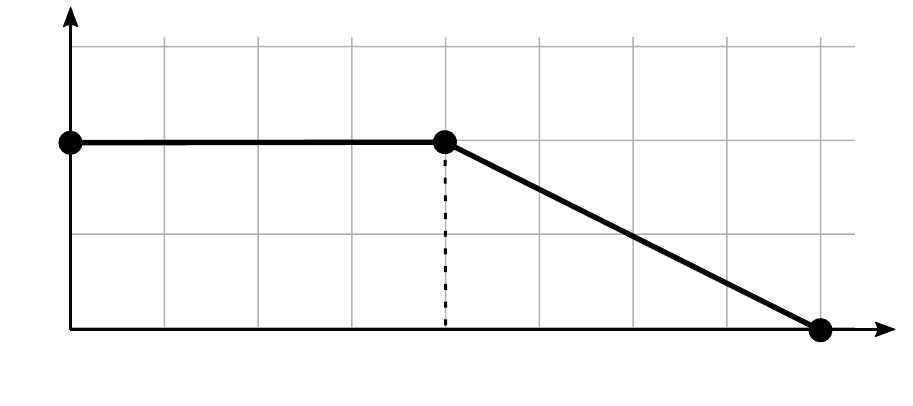
		\caption{The lattice {points} enclosed by the trapezoid can be obtained by adding up those in the rectangle and the triangle and subtracting the repeated ones on the dotted line.}\label{fig: lattice2}	    
		\end{subfigure}

	\caption{Convex generators $e_{p,q}^d$ and $e_{1,0}^{kd}e_{m,1}^d$.}

	\end{figure}
	
	    \end{proof}
	
	\subsection{Combinatorics of ECH generators with fixed endpoints}\label{section: nature}

	In this section, for a generator $\Lambda$ with fixed $x(\Lambda)$ and $y(\Lambda)$, we provide relations that $I(\Lambda)$ must satisfy.  These inequalities will be used to determine our obstructions to symplectic embeddings of polydisks into ellipsoids.
	The arguments presented below are elementary lattice point counts, which boil down to the convexity requirements as stipulated in Definition \ref{defn: convexpath}. Recall that the function a convex generator represents is \emph{concave} and has non-positive slope on each segment.

	
	\begin{lemma}\label{lemma: convexity}
	Let $x_0, y_0$ be positive integers and $\Lambda$ be a convex generator with $x(\Lambda) = x_0$ and $y(\Lambda) = y_0$. Then
	\begin{enumerate}[\em (i)]
	    \item $I(\Lambda) \le I(e_{1,0}^{x_0}e_{0,1}^{y_0})$.
\item	If, in addition, all edges of $\Lambda$ are labeled `$e$', then
         $I(\Lambda) \ge I(e_{x_0, y_0})$.
    \end{enumerate}
	
	\end{lemma}
	
\begin{proof}Let $x_0, y_0$ be positive integers and $\Lambda$ be a convex generator with $x(\Lambda) = x_0$ and $y(\Lambda) = y_0$.
\begin{enumerate}[(i)]
  \item  By the Definition \ref{defn: convexpath}, the graph of $\Lambda$ must be contained in the area enclosed by the graphs of $e_{x_0, y_0}$ and $e_{1,0}^{x_0}e_{0,1}^{y_0}$ (see Figure \ref{fig: convexity}), thus
   \[
L(e_{x_0, y_0}) \le L(\Lambda) \le L(e_{1,0}^{x_0}e_{0,1}^{y_0}).
     \]
     
It follows that 
\[
I(\Lambda) = 2(L(\Lambda) - 1) - h(\Lambda) \le 2(L(\Lambda) - 1) \le 2(L(e_{1,0}^{x_0}e_{0,1}^{y_0}) - 1) = I(e_{1,0}^{x_0}e_{0,1}^{y_0}).
\]

	\begin{figure}[H]
    \centering
		\def\svgwidth{0.55\columnwidth}
		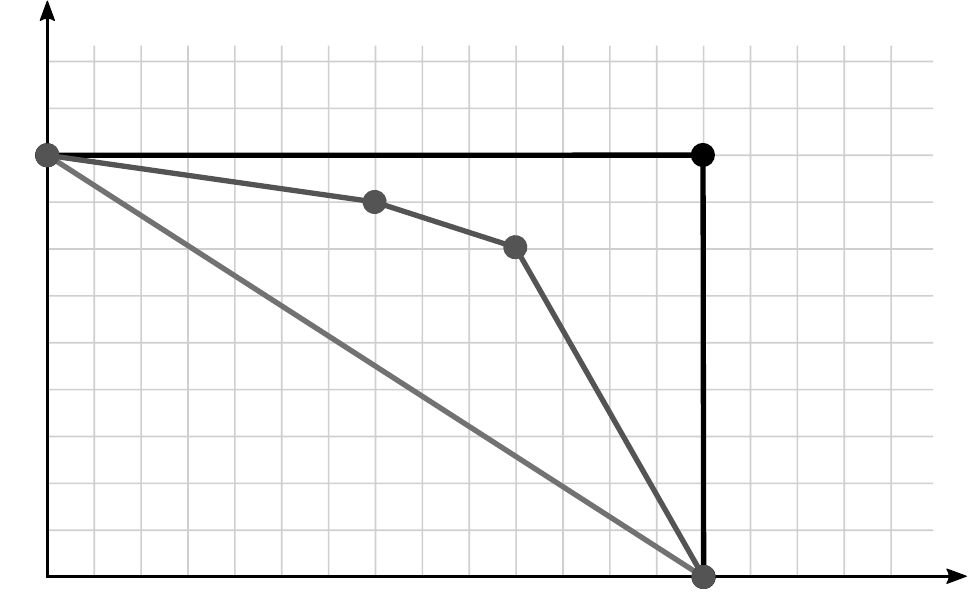
		\caption{Convex generators $e_{x_0, y_0}$, $e_{1,0}^{x_0}e_{0,1}^{y_0}$, and $\Lambda$ with $x(\Lambda) = x_0$ and $y(\Lambda) = y_0$.}\label{fig: convexity}
	\end{figure}
	
\item If we assume that $\Lambda$ is purely elliptic, i.e. that all edges of $\Lambda$ are labeled `$e$', then $h(\Lambda) = 0$ and we have
	\[
	I(e_{x_0, y_0}) = 2(L(e_{x_0, y_0}) - 1) \le 2(L(\Lambda) - 1) = I(\Lambda).
	\]
	\end{enumerate}
	\end{proof}

	

	\begin{lemma}\label{lemma: commonfactor}
	Let {$x_0, y_0$} be positive integers and $\Lambda$ be a convex generator with $x(\Lambda) = x_0$ and $y(\Lambda) = y_0$. If $\Lambda$ does not contain an $e_{1,0}$ factor, then $I(\Lambda) \le I(e_{x_0,1}e_{0,1}^{y_0-1}).$
	\end{lemma}
	
	\begin{proof}
	Let $x_0, y_0$ and $\Lambda$ be as given. Denote $k$ to be the slope of the first linear segment (that which intersects with the $y$-axis) of $\Lambda$. Since $\Lambda$ contains no $e_{0,1}$ factor, $k\ne0$. If $-\frac{1}{x_0}< k<0$, then since each linear segment of $\Lambda$ must have {endpoints with integer coordinates}, we must have $x(\Lambda)>x_0$, which is impossible. Therefore, we see that $k\le -\frac{1}{x_0}$, which implies that the set of lattice points enclosed by $\Lambda$ must be a subset of that enclosed by $e_{x_0,1}e_{0,1}^{y_0-1}$ by convexity (see Figure \ref{fig: convexity2}). 
		\begin{figure}[H]
    \centering
		\def\svgwidth{0.55\columnwidth}
		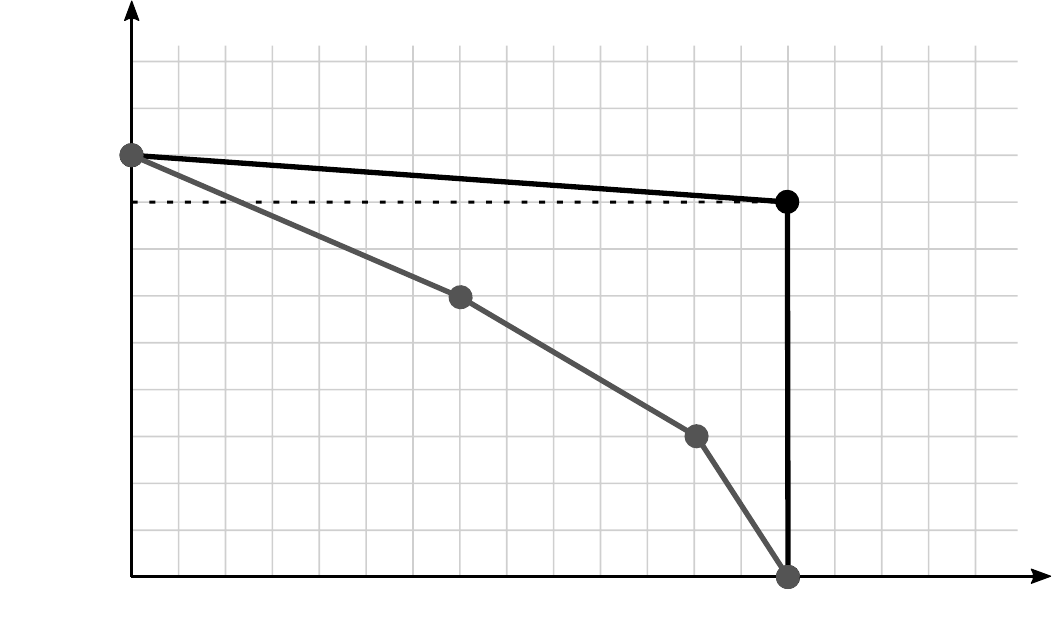
		\caption{Convex generators $e_{x_0,1}e_{0,1}^{y_0-1}$ and $\Lambda$ with $x(\Lambda) = x_0$ and $y(\Lambda) = y_0$ that does not contain an $e_{1,0}$ factor.}\label{fig: convexity2}
	\end{figure}

	Hence it follows from $L(\Lambda) \le L(e_{x_0,1}e_{0,1}^{y_0-1})$ that
	\[
	I(\Lambda)= 2(L(\Lambda)-1)- h(\Lambda) \le 2(L(e_{x_0,1}e_{0,1}^{y_0-1})-1) = I(e_{x_0,1}e_{0,1}^{y_0-1}),
	\]
	as desired.
	\end{proof}
	

	\subsection{Restrictions  from action and $J$-holomorphic curve genus}\label{section: restriction}
	
	In this section, we study convex generators $\Lambda$ such that $\Lambda \le_{P(a,1), E(pc/2, c)} e_{p,2}^d$ for some integer $d \ge 1$. The action inequality and the $J$-holomorphic curve genus inequality of Definition \ref{defn: Jcurve} allow us to find restrictions on $x(\Lambda)$ and $y(\Lambda)$ in relation to $a$, $p$, and $d$. These inequalities, when used in combination with the ECH index requirement, are crucial to providing obstructions through the Hutchings criterion.

	We first prove a more general result for rational $b = p/q$ and $\Lambda \le_{P(a, 1), E(bc,c)} e_{p,q}^d$.
	
	\begin{proposition}\label{prop: prop1}
		Let $a\ge 1$, $c>0$ and $b=p/q$ for $p, q$ relatively prime integers. Let $\Lambda$ be a convex generator. Suppose $P(a,1) \xhookrightarrow{s} E(bc, c)$ satisfies $pc< qa + p$. If $\Lambda \le e_{p,q}^d$ for some $d\ge 1$ then $y(\Lambda) < qd$.
	\end{proposition}
	
	\begin{proof} 
From the $J$-holomorphic curve genus inequality of Definition \ref{defn: Jcurve}, we have
\[
x(\Lambda) +y(\Lambda) \ge pd+qd.
\]
Suppose $y(\Lambda)\ge qd$.  Then the the action inequality of Definition \ref{defn: Jcurve} gives
\[
    pd+aqd \le x(\Lambda) + y(\Lambda) +(a-1)y(\Lambda) = A_{P(a,1)}(\Lambda) \le A_{E(cp/q, c)}(e_{p,q}^d) = pcd,
\]
    which is a direct contradiction to $pc<qa + p$.
    \end{proof}
	
	With this result, we are ready to derive the inequalities which will be central to our proofs of the main results.
	
	\begin{lemma}\label{lemma: lemma1}
		Let $a\geq 1$, $c>0$ and $b = p/2$ for $p>2$ some odd integer. Let $\Lambda$ be a convex generator. Suppose $P(a,1) \xhookrightarrow{s} E(bc, c)$ satisfies $pc< 2a + p$.  If $\Lambda\leq e^d_{p,2}$ then $y(\Lambda)<2d$ and
		\[
		a> \frac{ x(\Lambda) - pd}{2d - y({\Lambda})} \ge \frac{3d-1-y(\Lambda)}{2d-y(\Lambda)}. 
		\]
	\end{lemma}
	
	\begin{proof}
    Proposition \ref{prop: prop1} immediately tells us that $y(\Lambda)<2d$. Multiplying our hypothesis $pc<2a+p$ by $d$ provides $pcd<2ad+pd$. The action inequality of Definition \ref{defn: Jcurve} provides:
    \[
    x(\Lambda) + ay(\Lambda) = A_{P(a,1)}(\Lambda) \le A_{E(bc,c)}(e_{p,2}^d) = pcd.
    \]
    Stringing inequalities provides $x(\Lambda)+ay(\Lambda)<2ad+pd$, and so $x(\Lambda)-pd<2ad-ay(\Lambda)$. Now, because $y(\Lambda)<2d$, we factor and divide to get
    \begin{equation}\label{eqn: action}
        a>\frac{x(\Lambda)-pd}{2d-y(\Lambda)}.
    \end{equation}
    The $J$-holomorphic curve genus inequality of Definition \ref{defn: Jcurve} also tells us that $x(\Lambda)+y(\Lambda)\geq(p+3)d-1$ which can be rewritten as $x(\Lambda)-pd\geq3d-y(\Lambda)-1$, and from this we conclude that
    \[
    a> \frac{ x(\Lambda) - pd}{2d - y({\Lambda})} \ge \frac{3d-1-y(\Lambda)}{2d-y(\Lambda)}.
    \]
    In particular, we have 
    \begin{equation}\label{eqn: Jcurve}
        a> \frac{3d-1-y(\Lambda)}{2d-y(\Lambda)}.
    \end{equation}
    \end{proof}


    \subsection{Elimination of the trivial factorization}\label{section: triv}

   The Hutchings criterion imposes a condition on each pair of factors $\Lambda_i$ and $\Lambda'_i$. In particular, it requires $\Lambda \le_{P(a,1), E(pc/2, c)} e_{p,2}^{d_0}$ when the factorization is trivial. Thus we wish to prove the non-existence of the convex generator $\Lambda$ such that $\Lambda \le e_{p,2}^{d_0}$ whenever $d_0 \ge 2$, $1 \le a \le (2d_0 -1)/d_0$ and $p \ge 4d_0+1$. Lemma \ref{lemma: lemma1} tells us $\Lambda \le e_{p,2}^{d_0}$ only if $y(\Lambda) < 2d_0$. We split the remaining possibilities in two cases: 
 \begin{enumerate} 
\item The case when $d_0 \le y(\Lambda) < 2d_0$  as in Figure \ref{fig: triv1};
\item The case when $0 \le y(\Lambda) < d_0$ as in Figure \ref{fig: triv2}. 
 \end{enumerate} 
In the first case, the relatively large value of $y(\Lambda)$ allows $\Lambda$ to fulfill the index requirement of Definition \ref{defn: Jcurve} with considerable flexibility on $x(\Lambda)$, which prevents the action inequality of Definition \ref{defn: Jcurve}  alone from yielding the desired result. We will instead appeal to (\ref{eqn: Jcurve}), namely that:
\[
       a> \frac{3d-1-y(\Lambda)}{2d-y(\Lambda)},
\]
 which is the additional information provided by the $J$-holomorphic curve genus inequality, to prove obstructions.
    
    In the second case, the smaller value of $y(\Lambda)$ forces $x(\Lambda)$ to be sufficiently large in order for $\Lambda$ to achieve the same index as $e_{p,2}^{d_0}$. In consequence, we can derive from Lemmas \ref{lemma: count} and \ref{lemma: convexity}, which give information on the ECH index of $\Lambda$, a restriction on $x(\Lambda)$. Combining this restriction with the inequality  (\ref{eqn: action}) derived from action,
    namely that
    \[
            a>\frac{x(\Lambda)-pd}{2d-y(\Lambda)},
    \]
    proves the nonexistence of such $\Lambda$.

    \begin{figure}[H]
	    \centering
	    \begin{subfigure}{0.45\textwidth}
		\def\svgwidth{\columnwidth}
		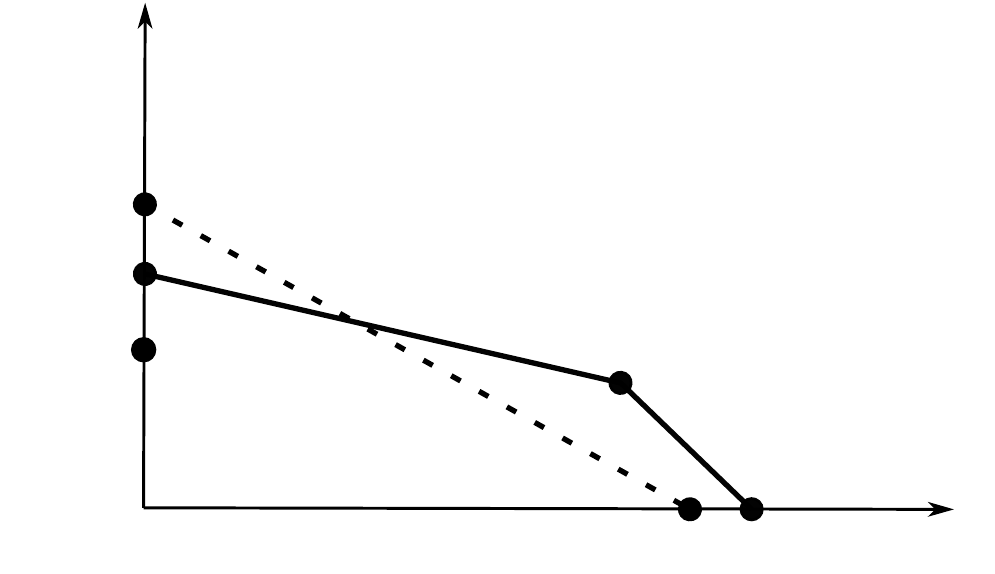
		\caption{$\Lambda \le e_{p,2}^{d_0}$ with $d_0 \le y(\Lambda) < 2d_0$.}\label{fig: triv1}
	    \end{subfigure}
	    \hspace*{3mm}
	    \begin{subfigure}{0.45\textwidth}
		\def\svgwidth{\columnwidth}
		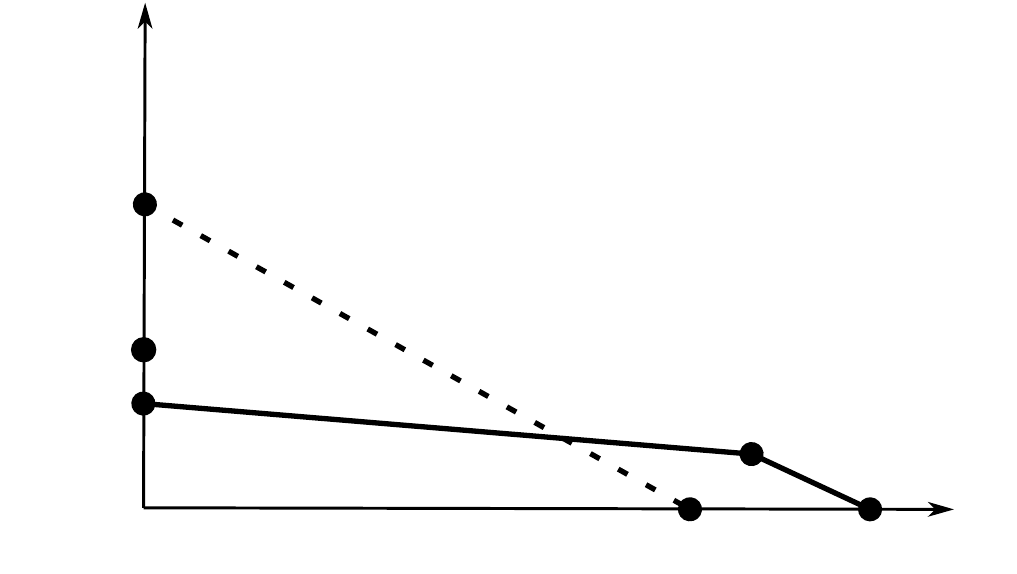
		\caption{$\Lambda \le e_{p,2}^{d_0}$ with $0 \le y(\Lambda) < d_0$.}\label{fig: triv2}
	    \end{subfigure}
	\caption{Different cases of $\Lambda \le e_{p,2}^{d_0}$.}
	\end{figure}
	
	In the proof of the following proposition, we will handle these two cases respectively through two claims. The result permits us to eliminate the trivial factorization of $\Lambda$ and also give information on $\Lambda \le e_{p,2}^d$ for $d \le d_0 - 1$ that will be useful later.
	
	\begin{proposition}\label{prop: trivfactor} 
	Let $d_0\ge 2$, $1\le a \le (2d_0-1)/d_0$, $c>0$ and $b=p/ 2$ for $p \ge 4d_0 +1$ an odd integer. Suppose $P(a,1) \xhookrightarrow{s} E(bc, c)$ satisfies $pc< 2a + p$. Then the following statements are true:
	\begin{enumerate}[\em (i)]
	    \item There exists no convex generator $\Lambda$ such that $\Lambda \le e_{p,2}^{d_0}$.
	    \item If $d \in [2, d_0 - 1]$ and $\Lambda$ is a convex generator such that $\Lambda \le e_{p,2}^d$, then $y(\Lambda) = d$.
	\end{enumerate}
	\end{proposition}
	\begin{proof}
	We will prove both (i) and (ii) using the following two claims.
	\begin{claim}\label{prop: trivfactor1}
	Let $d_0\ge 2$, $1\le a \le (2d_0 -1)/ d_0$, $c>0$ and $b=p/ 2$ for $p>2$ an odd integer.  Suppose $P(a,1) \xhookrightarrow{s} E(bc, c)$ satisfies $pc< 2a + p$. If $\Lambda$ is a convex generator for $P(a,1)$  such that 
		\begin{itemize}
			\item $\Lambda \le e_{p,2}^d$ and $d \in [2, d_0-1]$, then $y(\Lambda)\le d$.
			
			\item $\Lambda \le e_{p,2}^d$ and $d=d_0$, then $y(\Lambda)\le d - 1$.
		\end{itemize}
	\end{claim}
	\begin{claim}\label{prop: trivfactor2} 
		Let $d_0\ge 2$, $1\le a \le (2d_0-1)/d_0$, $c>0$ and $b=p/ 2$ for $p \ge 4d_0 +1$ an odd integer. Let $\Lambda$ be a convex generator. Suppose $P(a,1) \xhookrightarrow{s} E(bc, c)$ satisfies $pc< 2a + p$. If $\Lambda \le e_{p,2}^d$ for any $d \in [2, d_0]$, then $y(\Lambda) \ge d$.
	\end{claim}
	\begin{proof}[Proof of Claim \ref{prop: trivfactor1}]\let\qed\relax
	The right hand side of $(\ref{eqn: Jcurve})$ is monotonically increasing in variable $y(\Lambda)$ on the interval $0  \le y(\Lambda) <2d$. If $2\le d \le d_0$, suppose for contradiction $y(\Lambda) \ge d+1$, then 
    \[
    a> \frac{2d-2}{d-1}=2,
    \]
    a contradiction. If $d = d_0$, suppose for contradiction that $y(\Lambda) \ge d= d_0$. Then $a>(2d_0-1)/d_0$,  a contradiction, as desired.
	\end{proof}
	\begin{proof}[Proof of Claim \ref{prop: trivfactor2}]\let\qed\relax
	Fix $2\le d \le d_0$ with $\Lambda \le e_{p,2}^d$, and suppose for contradiction that $y(\Lambda) \le d-1.$ By Lemma \ref{lemma: convexity}(i) and Lemma \ref{lemma: count}(i), (v) we have
    \[
    I(\Lambda) = I(e_{p,2}^d) = 2pd^2 + (p+3)d \le 2(x(\Lambda)+y(\Lambda) +x(\Lambda)y(\Lambda))= I(e_{1,0}^{x(\Lambda)} e_{0,1}^{y(\Lambda)}).
    \]
    Rewriting gives
    \[
    x(\Lambda) \ge \frac{(2d^2+d)p + 3d - 2y(\Lambda)}{2y(\Lambda)+2}.
    \]
    Combining this with (\ref{eqn: action}), we get
    \[
    a > \frac{(2d^2 - d - 2dy(\Lambda))p + 3d - 2y(\Lambda)}{(2d-y(\Lambda))(2y(\Lambda)+2)}.
    \]
    One can verify that the right hand side is monotonically decreasing with respect to the variable $y(\Lambda)$ on the interval $0 \le y(\Lambda) \le d-1$ whenever $d\ge 2$ and $p\ge 4d_0+1$. It follows that the lowest bound on $a$ is achieved at $y(\Lambda)=d-1$, thus
    \[
    a > \frac{(p+1)d+2}{2d^2+2d}:=F(d).
    \]
    But note now that $F$ is monotonically decreasing with respect to the variable $d$. Therefore $F(d)> F(d_0)$ since $d\le d_0$. Finally plugging in $p \ge 4d_0 +1 \ge 4d_0 +1 - 4/d_0 $, yields
    the desired contradiction:    \[
    a > F(d_0) \ge \frac{2d_0^2+d_0-1}{d_0^2+d_0} = \frac{2d_0-1}{d_0}.
    \]
	\end{proof}
	
	Now, under the hypothesis of the statement of Proposition \ref{prop: trivfactor}, if $d = d_0$ and $\Lambda$ is a convex generator such that $\Lambda \le e_{p,2}^d$, then Claim \ref{prop: trivfactor1} tells us that $y(\Lambda) \le d -1$, but Claim \ref{prop: trivfactor2} implies $y(\Lambda) \ge d$. This is a contradiction, proving (i). On the other hand, if $d \in [2, d_0 - 1]$ and $\Lambda$ is a convex generator such that $\Lambda \le e_{p,2}^d$, then Claim \ref{prop: trivfactor1} and Claim \ref{prop: trivfactor2} show $d \le y(\Lambda) \le d$. This proves (ii). 
	\end{proof}
	
	\begin{remark}
	Proposition \ref{prop: trivfactor}(i) provides a sufficient condition for when the trivial factorization of $\Lambda$ is impossible, and Proposition \ref{prop: trivfactor}(ii) further restricts the remaining possibilities of $\Lambda$ satisfying $\Lambda\le e_{p,2}^d$ for $d \in [2, d_0 - 1]$, which we turn our attention to in the next section.
    \end{remark}

	\subsection{Elimination of the general factorization}\label{section: gen}

	We now aim to eliminate the possibility of $\Lambda$ having the general factorization, that is, $\Lambda = \prod_{i = 1}^{n} \Lambda_i$ for some $2 \le n \le d_0 - 1$ satisfying the Hutchings criterion, Theorem \ref{thm: hu criterion}. Corresponding to this factorization of $\Lambda$, we would have $\Lambda' = e_{p,2}^{d_0} = \prod_{i = 1}^{n} \Lambda'_i$, where $\Lambda'_i = e_{p,2}^{d_i}$ for each $i \in \{1, \dots,n\}$.
	
	Under our hypothesis that $d_0$ is an odd prime number, not all $\Lambda'_i$ can be the same in this factorization. On the other hand, the second condition of the Hutchings criterion forces $\Lambda'_i$ and $\Lambda'_j$ to be the same whenever $\Lambda_i$ and $\Lambda_j$ {share} a common factor of the form $e_{a,b}$.  
	
	In the proof of Proposition \ref{prop: genfactor} below, we use this observation to arrive at a contradiction, which allows us to eliminate the possibility of the general factorization.
	
	\begin{proposition}\label{prop: genfactor}
	    Let $d_0 \ge 3$ be a prime number. Let $1\le a \le (2d_0-1)/d_0$, $c>0$ and $b=p/ 2$ for $p \ge 4d_0 +1$. Suppose $P(a,1) \xhookrightarrow{s} E(bc, c)$ satisfies $pc< 2a + p$. If there exists a convex generator $\Lambda$, positive integer $1\le n \le d_0$ and factorizations $\Lambda = \prod_{i=1}^n \Lambda_i$ and $e^{d_0}_{p,2} = \prod_{i=1}^n e^{d_i}_{p,2}$ satisfying the three conditions in the Hutchings criterion, Theorem \ref{thm: hu criterion}, then $n \in \{1, d_0\}$.
	\end{proposition}
	\begin{proof}
	We first prove the following claim:
	\begin{claim}\label{claim: genfactor}
	Let $d \ge 1$, $1\le a \le 2$, $c>0$ and $b=p/2$ for an odd integer $p > 4d - 3$. Let $\Lambda$ be a convex generator. Suppose $P(a,1) \xhookrightarrow{s} E(bc, c)$ satisfies $pc< 2a + p$.  If $\Lambda \le e_{p,2}^d$ and $y(\Lambda)=d$, then $\Lambda$ contains an $e_{1,0}$ factor.
	\end{claim}
	\begin{proof}[Proof of Claim \ref{claim: genfactor}]\let\qed\relax
	If $\Lambda \le e_{p,2}^d$ with $y(\Lambda)=d$, note first that the $J$-holomorphic curve genus inequality of Definition \ref{defn: Jcurve}  gives $x(\Lambda) +y(\Lambda) \ge (p+3)d-1$ hence $x(\Lambda) \ge (p+2)d-1$. The action inequality of Definition \ref{defn: Jcurve} gives
    \[
    x(\Lambda)+ ad = x(\Lambda)+ ay(\Lambda) \le pcd < 2ad+ pd,
    \]
hence $x(\Lambda)< (p+a)d\le (p+2)d$, so $x(\Lambda)\le (p+2)d-1$ since $(p+2)d$ is an integer. Thus we must have $x(\Lambda) = (p+2)d-1$.

Now suppose for contradiction $\Lambda$ contains no $e_{1,0}$ factor. By Lemma \ref{lemma: commonfactor},

\[
I(\Lambda) \le I(e_{x(\Lambda),1}e_{0,1}^{y(\Lambda)-1}) = I(e_{(p+2)d-1,1}e_{0,1}^{d-1}).
\]

By the index requirement of Definition \ref{defn: Jcurve}, $I(\Lambda)  =  I(e_{p,2}^d)$. Combining this with the inequality above and the index formulae of Lemma \ref{lemma: count}(ii), (v) gives

    \[
    2pd^2+4d^2 = I(e_{(p+2)d-1,1}e_{0,1}^{d-1}) \ge I(\Lambda)  =  I(e_{p,2}^d) = 2pd^2 + (p+3)d,
    \]
    
which implies that $p\le 4d-3$, a contradiction.
	\end{proof}
	Let $d_0, a, p$ be as given in the statement of Proposition \ref{prop: genfactor}. Note that if $d = d_i$ for any $i \in \{1, \dots, n\}$, then $a, d, p$ also satisfy the hypothesis of Claim \ref{claim: genfactor}. Because the factorizations of $\Lambda$ and $e_{p,2}^{d_0}$ satisfy the Hutchings criterion, $\Lambda_i \le e_{p,2}^{d_i}$ for all $i \in \{1, \dots, n\}$. If $d_i \ge 2$, then by Proposition \ref{prop: trivfactor}(ii), $y(\Lambda_i) = d_i$, which implies that $\Lambda_i$ must have an $e_{1,0}$ factor by Claim \ref{claim: genfactor}. If $d_i=1$, then either $y(\Lambda_i)=0$, where $\Lambda_i=e_{1,0}^r$ for some integer $r$, or $y(\Lambda_i)=1$, which again by Claim \ref{claim: genfactor} implies that $e_{1,0}$ is a factor of $\Lambda_i$. We conclude that all $\Lambda_i$ share an elliptic orbit in common, hence in particular $d_i=d_j$ for every $i, j \in \{1, \dots, n\}$. This happens only if $n$ divides $d_0$. But $d_0 \ge 3$ is prime, so $n$ could only be $1$ or $d_0$, i.e., the general factorization is impossible.
	\end{proof}

	\subsection{Elimination of the full factorization}\label{section: full}

	Another possible outcome of applying the Hutchings criterion to $\Lambda' = e_{p,2}^{d_0}$ is that we obtain a $\Lambda$ and factorizations $\Lambda' = e_{p,2} \cdots e_{p,2}$ and $\Lambda = \prod_{i = 1}^{d_0}\Lambda_i$ that fulfill the three requirements of Theorem \ref{thm: hu criterion}. In particular, each $\Lambda_i$ should satisfy $\Lambda_i \le_{P(a,1), E(pc/2, c)} e_{p,2}$, and $I(\Lambda_i\Lambda_j) = I(e_{p,2}^2)$ for all $i, j \in \{1, \dots, d_0\}$. 
	
	We prove below the non-existence of such $\Lambda$ by showing that these two conditions cannot be satisfied at the same time under our hypothesis. 
	
	\begin{proposition}\label{prop: fullfactor} 
		Let $d_0 \ge 3$ be given. Let $a\ge 1$, $c>0$ and $b=p/2$ for $p>2$ an odd integer. Suppose $P(a,1) \xhookrightarrow{s} E(bc, c)$ satisfies $pc< 2a + p$. If there exists a convex generator $\Lambda$, positive integer $1\le n \le d_0$ and factorizations $\Lambda = \prod_{i=1}^n \Lambda_i$ and $e^{d_0}_{p,2} = \prod_{i=1}^n e^{d_i}_{p,2}$ satisfying the three conditions in the Hutchings criterion, Theorem \ref{thm: hu criterion}, then $n\ne d_0$.
	\end{proposition}
	
	\begin{proof}
	Suppose for contradiction $n=d_0$, we have $\Lambda_i \le e_{p,2}^1$ and $I(\Lambda_i\Lambda_j)=I(e_{p,2}^2)$ for all $i, j \in \{1, \dots, d_0\}$. By Proposition \ref{prop: prop1}, $y(\Lambda_i)$ can either be 0 or 1.   Thus we have three possibilities:

    \begin{enumerate}
	\item 
	At least $2$ of the $y(\Lambda_i)$ are $0$. Say $y(\Lambda_1) =y(\Lambda_2)=0$, hence $\Lambda_1=\Lambda_2 = e_{1,0}^{(3p+3)/2}$. Then
	\[
	I(\Lambda_1\Lambda_2) = 6p + 6 \ne 10 p + 6 = I(e_{p,2}^2)
	\]
	using Lemma \ref{lemma: count}(ii). This is a contradiction.
	
	\item 
	Only one of the $y(\Lambda_i)$ is $0$. Say $y(\Lambda_1)=0$ and $y(\Lambda_2)=y(\Lambda_3)=1$, hence $\Lambda_1 = e_{1,0}^{(3p+3)/2}$ and $\Lambda_i = e_{1,0}^{k_i}e_{m_i,1}$ with $4k_i+2m_i=3p+1$ for $i=2,3$. We must have $k_2=k_3=0$, otherwise for either $i=2$ or $i=3$, $\Lambda_1=\Lambda_i$ since they share the elliptic orbit $e_{1,0}$, contradicting $y(\Lambda_i) =1$. This forces $m_2=m_3=(3p+1)/2$ hence $\Lambda_2=\Lambda_3$. Then    using Lemma \ref{lemma: count}(ii) we obtain
	\[
	I(\Lambda_2\Lambda_3) = 9p+7 = 10p + 6 = I(e_{p,2}^2).
	\]
 This implies that $p=1$, a contradiction.
	
	\item 
	Assume $y(\Lambda_i) = 1$ for all $i \in \{1, \dots, d_0\}$, hence $\Lambda_i = e_{1,0}^{k_i}e_{m_i, 1}$ with $4k_i+2m_i=3p+1$ for $1\le i \le d_0$. 
	
	If $k_i= k_j= 0$ for $i\ne j$, then $\Lambda_i=\Lambda_j= e_{m,1}$ where $m=(3p+1)/2$, and the computation in case 2 shows $I(\Lambda_i\Lambda_j)=I(e_{p,2}^2)$ implies $p=1$, a contradiction.
	
	If $k_i\ne 0$ and $k_j\ne 0$ for $i\ne j$, then $\Lambda_i=\Lambda_j=e_{1,0}^ke_{m,1}$ as they share the elliptic orbit $e_{1,0}$. Then
	\[
	I(\Lambda_i\Lambda_j) = 12k+6m+4= 9p+7 = 10p+6 = I(e_{p,2}^2)
	\]
    using Lemma \ref{lemma: count} (ii). 
    Again this implies $p=1$, a contradiction. By the pigeonhole principle, these two cover all the cases when $y(\Lambda_i)=1$ for all $i\le d_0$ since $d_0\ge 3$.
    \end{enumerate}
	\end{proof}
	
	\subsection{Proofs of the main results}\label{section: proof}
	We are ready to present the proof of Theorem \ref{thm: NY1}.
	\begin{proof}[Proof of Theorem \ref{thm: NY1}]
	As in the theorem statement, let $d_0 \ge 3$ be a prime number. Let $1 \le a \le (2d_0 -1)/d_0$, $c > 0$ and $b = p/2$ for some odd integer $p \ge 4d_0+1$. Suppose instead $P(a,1) \xhookrightarrow{s} E(bc, c)$ with $pc < 2a + p$, i.e. the embedding is not a trivial inclusion. 
	
	
    


    We apply the Hutchings criterion, Theorem \ref{thm: hu criterion}, to the minimal convex generator $e_{p,2}^{d_0}$ of $E(bc, c)$ to obtain $\Lambda$, a positive integer $n\le d_0$, and factorizations $\Lambda = \prod_{i=1}^n \Lambda_i$ and $e^{d_0}_{p,2} = \prod_{i=1}^n e^{d_i}_{p,2}$ satisfying the three conditions of the Hutchings criterion. 
    
    By Proposition \ref{prop: trivfactor}(i), $\Lambda \le e_{p,2}^{d_0}$ is impossible, so $n \ne 1$. 
    
    By Proposition \ref{prop: fullfactor}, $n \ne d_0$. 
    
    By Proposition \ref{prop: genfactor}, however, $n$ can only be $1$ or $d_0$. This is a contradiction. 
    
    Therefore, $\Lambda$ does not exist and we must have $pc > 2a + p$, or $a + b \le bc$.
    \end{proof}


	
	
	As stated previously, the next two theorems, which we shall now prove, provide obstructions when $p$ is smaller.
	
	\begin{proof}[Proof of Theorem \ref{thm: NY2}]
    As in the theorem statement, let $1 \le a \le 4/3$, $c > 0$ and $b = p/2$ for some odd integer $p > 2$. Suppose instead $P(a,1) \xhookrightarrow{s} E(bc, c)$ with $2a + p > pc$. Let $d_0 \ge 2$ be given. Suppose $\Lambda \le e_{p,2}^d$ for an arbitrary integer $2 \le d \le d_0$. We will use Lemma \ref{lemma: lemma1}  to show that no such $\Lambda$ exists. 

    If $y(\Lambda) = 0$, then $\Lambda = e_{1,0}^{pd^2 + (p+3)d/2}$ by index computation. Inserting 
    \[
    x(\Lambda) = pd^2 + (p+3)d/2
    \]
     in (\ref{eqn: action}) gives, since $p\ge 3$ and $d\ge 2$,
    \[
    a > \frac{(2d-1)p + 3}{4} \ge 3.
    \]

    If $y(\Lambda) > 0$, by Proposition \ref{prop: prop1} we know that $y(\Lambda) < 2d$. Then since (\ref{eqn: Jcurve}) is monotonically decreasing in $y(\Lambda)$ for any $d\ge 2$, plugging in $y(\Lambda) = 1$ gives the lowest bound:
    \[
    a > \frac{3d - 2}{2d - 1} \ge \frac{4}{3},
    \]
    for any $d \ge 2$. We now apply the Hutchings criterion, Theorem \ref{thm: hu criterion} to the minimal convex generator $e_{p,2}^{d_0}$ for $E(bc, c)$ with $d_0 = 3$ to obtain $\Lambda$, a positive integer $n\le d_0$, and factorizations $\Lambda = \prod_{i=1}^n \Lambda_i$ and $e^{d_0}_{p,2} = \prod_{i=1}^n e^{d_i}_{p,2}$ satisfying the Hutchings criterion. By the above argument, $n$ cannot be $1$ or $2$. Also by Proposition \ref{prop: fullfactor}, $n\ne d_0$. Thus no such $\Lambda$ exists, a contradiction.
    \end{proof}


    \begin{proof}[Proof of Theorem \ref{thm: NY3}]
    As in the theorem statement, let $1 \le a \le 3/2$, $c > 0$ and $b = p/2$ for some odd integer $p \ge 7$. Suppose instead $P(a,1) \xhookrightarrow{s} E(bc, c)$ with $2a + p > pc$. Let $d_0 \ge 2$ be given. Suppose $\Lambda \le e_{p,2}^d$ for an arbitrary integer $2 \le d \le d_0$.  We will use Lemma \ref{lemma: lemma1}  to show that no such $\Lambda$ exists. 

    If $y(\Lambda) = 0$, then $\Lambda = e_{1,0}^{pd^2 + (p+3)d/2}$. Since $p\ge 3$ and $d\ge 2$, substituting $x(\Lambda) = pd^2 + (p+3)d/2$ in (\ref{eqn: action}) gives
    \[
    a > \frac{(2d-1)p + 3}{4} \ge 3.
    \]
If $2d> y(\Lambda)\ge 2$, then since (\ref{eqn: Jcurve}) is monotonically decreasing in $y(\Lambda)$ for any $d\ge 2$, plugging in $y(\Lambda)=2$ gives the lowest bound:
    \[
    a> \frac{3d-3}{2d-2} = \frac{3}{2}.
    \]

    Finally, if $y(\Lambda) =1$, by Lemma \ref{lemma: convexity}(i), 
    \[
    I(\Lambda) = I(e_{p,2}^d) = 2pd^2 + (p+3)d \le 4x(\Lambda) +2,
    \]
and
    \begin{equation}\label{eqn: xineq}
        x(\Lambda) \ge \frac{2pd^2 + (p+3)d - 2}{4}.
    \end{equation}
Plugging (\ref{eqn: xineq}) and $y(\Lambda) = 1$ in (\ref{eqn: action}), we get
    \[
    a> \frac{(2d^2-3d)p+3d-2}{8d-4}\ge \frac{7d^2-9d-1}{4d-2} \ge \frac32,  
    \]
since $p\ge 7$ and the function is monotonically increasing in variable $d$ when $d\ge 2$. We now apply the Hutchings criterion, Theorem \ref{thm: hu criterion} to the minimal convex generator $e_{p,2}^{d_0}$ of $E(bc, c)$ with $d_0 = 3$ to obtain $\Lambda$, a positive integer $n\le d_0$, and factorizations $\Lambda = \prod_{i=1}^n \Lambda_i$ and $e^{d_0}_{p,2} = \prod_{i=1}^n e^{d_i}_{p,2}$ satisfying the Hutchings' criterion. By the above argument, $n$ cannot be $1$ or $2$. Also by Proposition \ref{prop: fullfactor}, $n\ne d_0$. Thus no such $\Lambda$ exists, a contradiction.    \end{proof}

\section{Prospects on extending Theorem \ref{thm: NY1} via the  Hutchings criterion}\label{bigsection3}

{In this section we discuss the scope of the Hutchings criterion as it pertains to obstructing symplectic embeddings of polydisks into ellipsoids.} In Section \ref{bigsection2}, we proved our main theorems by checking the Hutchings criterion, Theorem \ref{thm: hu criterion}, against the minimal convex generator $e_{p,q}^{d_0}$ for $E(bc, c)$, where $b = p/q$ with $q = 2$ and $p$ odd, which gave obstructions to symplectic embeddings of $P(a,1)$ into $E(bc, c)$. In Section \ref{sec:limitations}, we show our method of using the minimal generator $e_{p, 2}^{d_0}$ for $E(bc,c)$ cannot provide obstructions in the same way if the restrictions on the $a$ and $p$ values are weakened respectively to $1 \le a \le (2d_0-1)/d_0 + \varepsilon$ for any positive $\varepsilon$ and $p \ge 4d_0 - 1$.  In Section \ref{sec:moreq}, we show that the use of the minimal generator $e_{p, q}^{d_0}$ cannot be used to obstruct embeddings of polydisks $P(a,1)$ into ellipsoids $E(bc,c)$ if  $b = p/q$ for any coprime integers $p > q \ge 3$.  

In particular, our Propositions \ref{prop: example1}-\ref{prop: example3} demonstrate that the restrictions on $a$ and $b$ in Theorem \ref{thm: NY1} are optimal with respect to the use of the minimal convex generator $e_{p, 2}^{d_0}$ for $E(pc/2, c)$ in the Hutchings criterion.  The proofs of Propositions \ref{prop: example1}-\ref{prop: example3} rely on certain combinatorics of convex generators, which we provide in  Section \ref{sec:arbind}. Of particular interest is Lemma \ref{lemma: gen example}, which encodes the combinatorial information of a given generator $\Lambda$ with respect to its index $I(\Lambda)$ and endpoint values $x(\Lambda)$, $y(\Lambda)$.  This lemma provides classes of abstract examples satisfying the Hutchings criterion, leading to the limitations we establish on the Hutchings criterion with respect to the minimal convex generator $e_{p, q}^{d_0}$. 

It remains to consider alternate minimal convex generators for the ellipsoid which realize more complex lattice paths.  In Section \ref{subsection:classification}, we investigate this possibility. 
Proposition \ref{prop: classification} provides an abstract description for all minimal convex generators for the ellipsoid $E(bc, c)$ for any real numbers $b \ge 1, c >0$, and Proposition \ref{prop: p=2classification} extracts more information for the case where $b$ is a half-integer and provides an explicit form for the minimal convex generators for $E(bc, c)$.  However, as explained in  Remark \ref{rem:othergens}, the combinatorial methods developed in this paper are inconclusive when applied to these more complicated convex generators, indicating an avenue for further research.


\subsection{Achieving arbitrary index through maximal generators}\label{sec:arbind}

To develop examples demonstrating the limitations of the Hutchings criterion when applied to the minimal convex generator $e_{p, 2}^{d_0}$ for the ellipsoid, it is key to first construct abstract examples of generators of arbitrarily large index.  We construct these examples by utilizing the ability to ``transform" non-integral convex paths to integral convex paths, as in Definition \ref{defn: convexpath}, so that they exactly enclose identical sets of lattice points.
	
\begin{definition}
Let $\Gamma$ be a convex path in $\mathbb{R}^2$. We say a convex integral path $\Lambda$ is \emph{maximal under} $\Gamma$ if $\Lambda$ encloses precisely all lattices points in the first quadrant enclosed by $\Gamma$, including those on $\Gamma$.
\end{definition}


The existence and uniqueness of maximal generators under a convex path is guaranteed when a certain integral condition is met:

\begin{lemma} \label{lemma: convex path}
Given a convex path $\Gamma$ such that each linear segment of $\Gamma$ passes through an integer lattice point, there exists a unique convex integral path $\Lambda$ that is maximal under $\Gamma$.
\end{lemma}

We note that a convex path $\Gamma$ satisfying the hypothesis of Lemma \ref{lemma: convex path} need not be a convex integral path. {The idea of the proof is to ensure that the convex hull of the set of points enclosed by $\Gamma$ can be enclosed by a convex integral path.}

\begin{proof}
Any convex integral paths $\Lambda$, $\Lambda'$ that are maximal under $\Gamma$ enclose the same set of lattice points, therefore uniqueness is evident from convexity. It remains to prove existence.

We explicitly construct such a convex integral path $\Lambda$ as follows. Let $n$ denote the maximal $y$-coordinate of lattice points enclosed by $\Gamma$, including those on the boundary. For all integers $0\le k \le n$, we pick the largest integer $x_k$ such that the lattice point $(x_k,k)$ in the first quadrant is enclosed by $\Gamma$. For every $k>0$, we may choose {a positive integer} $k'$ so that $0\le k'<k$ and the slope $m$ of the line joining $(x_k, k)$ and $(x_{k'},k')$, is the largest amongst those obtained from joining $(x_k,k)$ with any other $(x_{k'},k')$.  Note that the slope $m$ can be negative infinity.

Set $\Lambda = e_{1,0}^{x_n}$. We then proceed inductively from $k=n$, where in each step we choose $k'$ as in the procedure above, and add an $e_{x_{k'}-x_k, k-k'}$ factor to $\Lambda$. We repeat this process starting from the new $k'$ and stop when $k'=0$. The process terminates in at most $n$ steps. Note that in each step, the slope of the new elliptic orbit added is always less than any previous ones, otherwise this would contradict maximality of the slope in the previous step. Thus the formal product $\Lambda$ is geometrically exactly the convex integral path connecting the chosen $(x_k, k)$. It follows that $\Lambda$ encloses all the lattice points enclosed by $\Gamma$. 
\end{proof}


Using the above procedure, we prove the following lemma by building a convex path subject to the lattice requirement. 

\begin{lemma}\label{lemma: gen example}
Let integers $x_0, y_0> 0$ be given. Let $L$ be an integer satisfying
    \[
    L_-:= L(e_{x_0,y_0}) \le L \le L(e_{1,0}^{x_0}e_{0,1}^{y_0}) =: L_+.
    \]
Then there exists a convex generator $\Lambda$ satisfying $x(\Lambda) =x_0$, $y(\Lambda)=y_0$ and $L(\Lambda)=L$.
\end{lemma}

\begin{proof}



Write $m= -y_0/x_0$. We denote $S$ to be the set of all lattice points $(x,y)$ in the first quadrant such that $(x,y)$ is enclosed by $e_{1,0}^{x_0}e_{0,1}^{y_0}$ but not by $e_{x_0,y_0}$. For each lattice point $(x,y)\in S$, there exists a unique line of slope $m$ passing through $(x,y)$, which we will denote $\eta(x,y)$. We put an ordering on $S$ by asserting that $(x_1,y_1)\preceq (x_2,y_2)$ if and only if 
    \begin{align*}
    &mx_1-y_1< mx_2-y_2  \\
    \text{ or }\ &mx_1-y_1= mx_2-y_2, \, x_1\le x_2
    \end{align*}
Geometrically, we are arranging the lattice points in $S$ into subclasses determined by the linear equation $\eta(x,y)$, among which we then sort using the $x$-coordinate. Intuitively, this ordering on $S$ gives us the order in which to ``add" points to the convex generator $e_{x_0,y_0}$, one at a time to maintain convexity, which we will now rigorously show.

By definition, it is clear that $\preceq$ is a strict total order. 
Note that $S$ contains $L_+-L_-$ distinct points. Thus there exists a unique order isomorphism from $(S,\preceq)$ to $[1, L_+-L_-]\cap \mathbb{Z}$ with the usual ordering. 

Now, let $(x', y')\in S$ be the element corresponding to $L-L_-$ via the order isomorphism. Consider $\eta(x', y')$, which may pass through multiple elements of $S$. We may rotate $\eta(x',y')$ clockwise about the point $(x',y')$, by a small angle, to obtain a new line $\eta'$, such that $\eta'$ encloses precisely every lattice on and under $\eta(x',y')$ in the first quadrant except for those both on the line $\eta(x',y')$ and to the right of $(x',y')$.  An example of this procedure is given in Figure \ref{fig: line}. 

\begin{figure}
    \centering
		\def\svgwidth{0.6\columnwidth}
		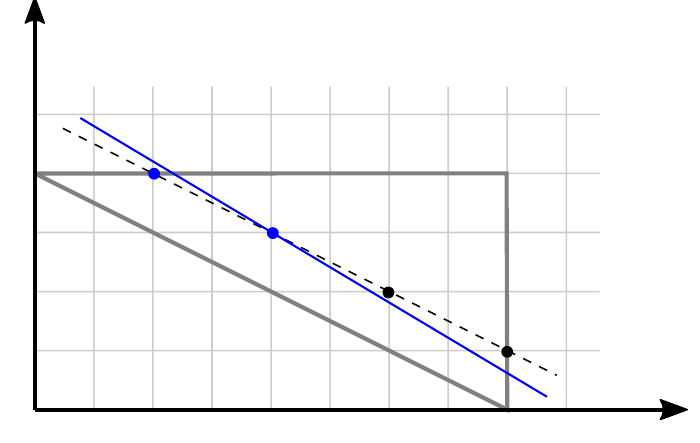
    		\caption{In this example, $x_0 = 8$ and $y_0 = 4$. Note that $(4,3) \in (S, \preceq)$ corresponds to $6$ under the given order isomorphism, and we can rotate $\eta(4,3)$ according to the described procedure to obtain $\eta'$, which encloses precisely $6$ points in $S$. } \label{fig: line}
\end{figure}

We can obtain the $\eta'$ in the way described above because $\mathbb{Z}^2$ is discrete and the number of lattice points enclosed by $\eta(x',y')$ in the first quadrant is finite. Note that the slope $m'$ of $\eta'$ satisfies $-\infty<m'<m<0$ by construction. We thus denote $\Delta$ the region in the first quadrant enclosed by $\eta'$, the vertical line $x=x_0$ and the horizontal line $y=y_0$, and we denote $\Gamma$ the convex path on the boundary $\partial\Delta$ removing those on the $x$, $y$-axes. By Lemma \ref{lemma: convex path}, there exists a purely elliptic convex generator $\Lambda$ that is maximal under $\Delta$. By maximality, $x(\Lambda)=x_0$, $y(\Lambda)=y_0$, and $\Lambda$ encloses precisely:
    \begin{itemize}
        \item Each of the {lattice} points under $e_{x_0,y_0}$;
        \item Each of the elements $(x,y)\in S$ satisfying $mx-y<mx'-y'$;
        \item Each the elements $(x,y)\in S$ satisfying $mx-y=mx'-y'$ and $x\le x'$.
    \end{itemize}
There are precisely $L_-=L(e_{x_0,y_0})$ elements in the first category, while there are precisely $L- L_-$ elements in $S$ preceding $(x',y')$, as in the second and third categories combined. We conclude that $L(\Lambda)=L$, as desired.
\end{proof}

\subsection{Limitations of $e_{p, 2}^{d_0}$ for obstructing polydisk embeddings into $E(pc/2, c)$} \label{sec:limitations}


We now address the limitations of our method by investigating two key steps in the proof of Theorem \ref{thm: NY1}, namely Claims \ref{prop: trivfactor1} and \ref{prop: trivfactor2} of Proposition \ref{prop: trivfactor}, where the restrictions $a\le (2d_0 -1)/d_0$ and $p\ge 4d_0+1$ naturally arise.

Claims \ref{prop: trivfactor1} and \ref{prop: trivfactor2} establish conditions on the existence of a trivial factorization coming from the Hutchings criterion in terms of certain requirements on $P(a,1)$ and $E(pc/2,c)$. We  first prove that if we extend the upper bound $a\le (2d_0-1)/d_0$ by any positive amount
, Claim \ref{prop: trivfactor1} no longer holds:

\begin{proposition}\label{prop: example1}
Let $\varepsilon>0$, $d_0\ge 2$, $a= (2d_0 -1)/ d_0+\varepsilon$, $c>0$, and $b=p/ 2$ for $p>2$ an odd integer. If we assume $2a+p-\varepsilon/2 < pc$, then there always exists a convex generator $\Lambda$ such that $\Lambda \le e_{p,2}^{d_0}$.
\end{proposition}

\begin{remark*}
Note that the hypothesis on $b, c, p , d_0$ in Proposition \ref{prop: example1} is the same as in Claim \ref{prop: trivfactor1} except that we changed $1\le a\le (2d_0-1)/d_0$ to $a=(2d_0-1)/d_0 +\varepsilon$. The inequality 
\[
2a+p-\varepsilon/2 <pc <2a+p,
\]
 corresponds to when the domain $P(a,1)$ does not trivially include into $E(pc/2,c)$.
\end{remark*}

\begin{proof}  

We claim that there exists a purely elliptic $\Lambda$ with $x(\Lambda) = (p+2)d_0 - 1$, $y(\Lambda) = d_0$, and 
\[
 I(\Lambda)= I(e_{p,2}^{d_0}) = 2pd_0^2 +(p+3)d_0.
\]

To see this, note first that since $d_0\ge 2$, we have ${\rm gcd}(x(\Lambda),y(\Lambda)) = 1$, so indeed using  Lemma \ref{lemma: count}(i) and (iii),
    \[
    (p+2)d_0^2 + (p+2)d_0 = I(e_{x(\Lambda),y(\Lambda)})\le  I(e_{p,2}^{d_0}) \le I(e_{1,0}^{x(\Lambda)}e_{0,1}^{y(\Lambda)}) = 2(p+2)d_0^2 + 2(p+2)d_0 -2,
    \]
    which implies
    \[
    L(e_{x(\Lambda),y(\Lambda)})\le  L(e_{p,2}^{d_0}) \le L(e_{1,0}^{x(\Lambda)}e_{0,1}^{y(\Lambda)}).
    \]
    By Lemma \ref{lemma: gen example}, this proves the existence of $\Lambda$. We now argue that $\Lambda\leq e^{d_0}_{p,2}$. To see that the $J$-holomorphic curve genus inequality of Definition \ref{defn: Jcurve} holds, note that
    \[
    x(\Lambda)+y(\Lambda) = (p+3)d_0-1 \ge (p+3)d_0-1.
    \]
Further, recall that $A_{E(pc/2,c)}(e_{p,2}^{d_0}) = pcd_0$, thus the action inequality of Definition \ref{defn: Jcurve} is satisfied since
    \[
    A_{P(a,1)}(\Lambda)=x(\Lambda) + ay(\Lambda) =(p+4)d_0 -2 + \varepsilon d_0 = 2ad_0 + pd_0 - \varepsilon d_0,
    \]
which holds by the hypothesis $2a+p-\varepsilon<pc$. This shows that $\Lambda\le e_{p,2}^{d_0}$, as desired.
\end{proof}

We similarly examine the conditions on $p$ mandated by Claim \ref{prop: trivfactor2}. We show that if we decrease the lower bound on $p\ge 4d_0 +1$ by taking $p=4d_0-3$, the second largest odd integer after $4d_0 + 1$, then Claim \ref{prop: trivfactor2} no longer holds.

\begin{proposition}\label{prop: example2}
Let $d_0\ge 2$, $a= (2d_0-1)/d_0$, $c>0$ and $b=p/2$ for $p=4d_0 -3$. If
    \[
    2a +p-\frac{d_0-1}{d_0^2}<pc,
    \]
then there always exists a convex generator $\Lambda$ such that $\Lambda \le e_{p,2}^{d_0}$.

\end{proposition}

\begin{remark*}
Note again that the hypothesis of Proposition \ref{prop: example2} is consistent with that of Claim \ref{prop: trivfactor2} except that $p=4d_0-3 < 4d_0+1$. 
\end{remark*}

\begin{proof}
We claim that there exists a purely elliptic $\Lambda$ with $x(\Lambda) = (p+2)d_0 $, $y(\Lambda) = d_0-1$, and 
\[
I(\Lambda)= I(e_{p,2}^{d_0}) = 2pd_0^2 +(p+3)d_0.
\]
 First note that ${\rm gcd}(x(\Lambda), y(\Lambda))\le 3$ since $x(\Lambda) = (4d_0 + 3)y(\Lambda) + 3 $, so indeed, using $p=4d_0 -3$ and Lemma \ref{lemma: count}(i) and (iii),
    \[
    I(e_{x(\Lambda),y(\Lambda)})\le  (p+2)d_0^2+d_0+2 \le  I(e_{p,2}^{d_0}) \le I(e_{1,0}^{x(\Lambda)}e_{0,1}^{y(\Lambda)})= 2(p+2)d_0^2 + 2d_0 -2.
    \]
    The inequalities in the hypothesis of Lemma \ref{lemma: gen example} are satisfied, proving the existence of such $\Lambda$. We now argue that $\Lambda\leq e^{d_0}_{p,2}$. The $J$-holomorphic curve genus inequality  of Definition \ref{defn: Jcurve} between $\Lambda$ and $e_{p,2}^{d_0}$ holds as
    \[
    x(\Lambda)+y(\Lambda) = (p+3)d_0-1 \ge (p+3)d_0-1.
    \]
Further, recall that $A_{E(pc/2,c)}(e_{p,2}^{d_0}) = pcd_0$, thus the action inequality of Definition \ref{defn: Jcurve} is satisfied since
    \[
    A_{P(a,1)}(\Lambda) = x(\Lambda) + ay(\Lambda) =pd_0 + 2d_0 +ad_0-a = 2ad_0 +pd_0 - \frac{d_0-1}{d_0^2}d_0,
    \]
which holds by hypothesis. This shows that $\Lambda\le e_{p,2}^{d_0}$, as desired.
\end{proof}


With the application of the Hutchings criterion in mind, Propositions \ref{prop: example1} and \ref{prop: example2} imply that the trivial factorization is always possible under their respective hypotheses. That is, no contradiction of any kind can be achieved to obstruct a nontrivial embedding $P(a,1)\xhookrightarrow{s} E(pc/2,c)$ satisfying respectively $2a+p-\varepsilon <pc<2a+p$ and $2a+p-\mathcal{O}(d_0^{-1})<pc<2a+p$. Thus, applying the Hutchings criterion, Theorem \ref{thm: hu criterion}, to the minimal generator $e_{p,2}^{d_0}$ for any $d_0\ge 2$ will not produce an obstruction for any $a>(2d_0-1)/d_0$ or $p< 4d_0+1$ beyond the stipulations of Theorem \ref{thm: NY1}.

\subsection{Limitations of $e_{p, q}^{d_0}$ for obstructing polydisk embeddings into $E(pc/q, c)$}\label{sec:moreq}

We now illustrate the difficulties in extending our result to rational $b = p/q$ using the combinatorial tools employed in the proof of our main theorems. Although Proposition \ref{prop: prop1} provides the general result that $\Lambda \le e_{p, q}^{d}$ is possible only when $y(\Lambda) < qd$, the increasing value of $q > 2$ prevents further restrictions on $y(\Lambda)$ from the action and $J$-holomorphic curve genus inequalities of Definition \ref{defn: Jcurve}. Thus we are unable to provide statements analogous to Proposition \ref{prop: trivfactor} for $q > 2$.

The following result shows that if we allow $b$ to be an arbitrary rational number, then the trivial factorization is always possible, hence no obstructions to nontrivial embeddings of $P(a, 1)$ into $E(bc, c)$ can be obtained through the Hutchings criterion when applied to $e_{p, q}^{d_0}$.

\begin{proposition}\label{prop: example3}
Let $d_0\ge 2$, $a = (2d_0 -1)/ d_0$, $c>0$, and $b=p/ q$ for $p > q > 3$ and $p, q$ coprime integers. If we assume 
    \[
    qa + p -\frac{(q-3)(d_0-1)}{2d_0} < pc,
    \]
then there always exists a convex generator $\Lambda$ such that $\Lambda \le e_{p,q}^{d_0}$.
\end{proposition}
\begin{remark*}
Note that we take $a$ to be its greatest value allowed by Theorem \ref{thm: NY1}, which is less than $2$. The inequality 
\[
qa+p-\frac{(q-3)(d_0-1)}{2d_0} <pc <qa+p
\]
 corresponds to when the domain $P(a,1)$ cannot trivially include into $E(pc/q,c)$. 
\end{remark*}
\begin{proof}
We claim that there exists a purely elliptic $\Lambda$ with $y(\Lambda) = \left\lceil\frac{q}{2}\right\rceil d_0$, $x(\Lambda) = (p+q+1)d_0 - 1 - \left\lceil\frac{q}{2}\right\rceil d_0$ such that $\Lambda \le e_{p,q}^{d_0}$. We prove the existence of such $\Lambda$ separately with two cases: $q$ even and $q$ odd.
\begin{enumerate}
    \item 
    Suppose $q$ is even. Then $y(\Lambda) = \frac{q}{2}d_0$ and $x(\Lambda) = pd_0 + \frac{q}{2}d_0 + d_0 - 1$. We first show the existence of $\Lambda$ with
    \[
    I(\Lambda) = I(e_{p,q}^{d_0}) = pqd_0^2+(p+q+1)d_0.
    \]
    By Lemma \ref{lemma: count}(i), 
    \[
    I(e_{1,0}^{x(\Lambda)}e_{0,1}^{y(\Lambda)}) = \left(pq+\frac{q^2}{2}+ q\right)d_0^2 + (q + 2p+ 2)d_0 - 1.
    \]
    Hence 
    \[
    I(e_{1,0}^{x(\Lambda)}e_{0,1}^{y(\Lambda)}) - I(e_{p,q}^{d_0}) = \left(\frac{q^2}{2} + q\right)d_0^2+(p+1)d_0-1 > 0,
    \]
    implying $I(e_{p,q}^{d_0}) < I(e_{1,0}^{x(\Lambda)}e_{0,1}^{y(\Lambda)})$. 
    
    Next, using Lemma \ref{lemma: count}(iii) and the fact that $ {\rm gcd} (x(\Lambda), y(\Lambda))\le y(\Lambda)$, we get
    \begin{align*}
    I(e_{x(\Lambda), y(\Lambda)}) &= x(\Lambda) +  y(\Lambda) + x(\Lambda)y(\Lambda) + {\rm gcd} (x(\Lambda), y(\Lambda)) \\
    &\le x(\Lambda) +  2y(\Lambda) + x(\Lambda)y(\Lambda)\\
    &= \left(\frac{pq}{2}+\frac{q^2}{4}+\frac{q}{2}\right)d_0^2+(p+q+1)d_0-1. 
    \end{align*}
    Subtracting $I(e_{p,q}^{d_0})$ from the quantity in the last line above we get 
    \[
    \frac{(-2pq+q^2+q)d_0^2-4}{4} < 0,
    \]
    implying $I(e_{x(\Lambda), y(\Lambda)}) < I(e_{p,q}^{d_0})$. It then follows from Lemma \ref{lemma: gen example} that $\Lambda$ exists. Now, the $J$-holomorphic curve genus inequality of Definition \ref{defn: Jcurve} holds because
    \[
    x(\Lambda) + y(\Lambda) = (p+q+1)d_0-1 \ge (p+q+1)d_0-1 = x(e_{p,q}^{d_0}) + y(e_{p,q}^{d_0}) + m(e_{p,q}^{d_0}) - 1.
    \]
    Further, recall that $A_{E(pc/q,c)}(e_{p,q}^{d_0}) = pcd_0$, thus the action inequality of Definition \ref{defn: Jcurve} is satisfied since
    \begin{align*}
         A_{P(a,1)}(\Lambda) &= x(\Lambda) + ay(\Lambda)\\
         &= qad_0 + pd_0 - \frac{(q-2)(d_0-1)}{2d_0}d_0\\
         &< qad_0 + pd_0 - \frac{(q-3)(d_0-1)}{2d_0}d_0 < pcd_0,
    \end{align*}
    where the last inequality holds by the hypothesis. Thus we have $\Lambda \le e_{p,q}^{d_0}$.

    \item 
    Suppose $q$ is odd. Then $y(\Lambda) = \frac{q+1}{2}d_0$ and $x(\Lambda) = pd_0 + \frac{q+1}{2}d_0 - 1$. The rest of the steps are similar to the previous case. One can check that 
    \[
    I(e_{x(\Lambda), y(\Lambda)}) \le I(e_{p,q}^{d_0}) \le I(e_{1,0}^{x(\Lambda)}e_{0,1}^{y(\Lambda)}). 
    \]
    Thus by applying Lemma \ref{lemma: gen example} we prove the existence of $\Lambda$ with these $x(\Lambda)$ and $y(\Lambda)$ values and $I(\Lambda) = I(e_{p,q}^{d_0})$. Again, the the $J$-holomorphic curve genus inequality of Definition \ref{defn: Jcurve} holds because
    \[
    x(\Lambda) + y(\Lambda) = (p+q+1)d_0-1 \ge (p+q+1)d_0-1.
    \]
    Finally, the action inequality of Definition \ref{defn: Jcurve} is satisfied since
    \[
    A_{P(a,1)}(\Lambda) = qad_0 + pd_0 - \frac{(q-3)(d_0-1)}{2d_0}d_0 < pcd_0 = A_{E(pc/q,c)}(e_{p,q}^{d_0}).
    \]
    We again conclude $\Lambda \le e_{p,q}^{d_0}.$
\end{enumerate}
\end{proof}
\begin{remark}
 Note that in Proposition \ref{prop: example3} we require $q > 3$ because in this case there is a nice general formulation of $\Lambda$ satisfying $\Lambda\le e_{p,q}^{d_0}$. One can easily construct explicit examples of $\Lambda \le e_{p, q}^{d_0}$ for $q = 3$ by modifying the value of $y(\Lambda)$.
\end{remark}

\subsection{{Classification of minimal generators for the ellipsoid}}\label{subsection:classification}

To understand the extent to which the Hutchings Criterion, Theorem \ref{thm: hu criterion}, applies to a symplectic embedding problem, it is necessary to understand the minimal convex generators for the target of interest. In this section, we classify all minimal convex generators for the ellipsoid $E(bc,c)$ in Proposition \ref{prop: classification}, where $b\ge 1$, $c>0$ are any real numbers. Proposition \ref{prop: classification} provides a converse to \cite[Lem.~2.1(a)]{H} and generalizes the result of \cite[Lem.~A.3]{CN}, which considers the case $b=1$. We then specialize to the $b=p/2$ case, and provide an explicit formula for the minimal generators of half integer ellipsoids in Proposition \ref{prop: p=2classification}.


Recall that a convex generator $\Lambda$ is minimal (Definition \ref{def: minimalgen}) for a convex toric domain $X_\Omega$ if $\Lambda$ is purely elliptic and $\Lambda$ \emph{uniquely} minimizes the symplectic action $A_\Omega$ among all purely elliptic convex generators of the same index.   Moreover, the symplectic action of minimal generators is related to ECH capacities as follows.
\begin{remark} 
If $I(\Lambda)=2k$ and $\Lambda$ is minimal for $X_\Omega$, then $A_\Omega(\Lambda) = c_k(X_\Omega)$, by \cite[Prop.~5.6]{H}.
\end{remark}

We obtain the following characterization of minimal generators for the ellipsoid:

\begin{proposition}\label{prop: classification}

Let $b\ge 1$, $c>0$. Let $\Lambda$ be any purely elliptic convex generator, and $\eta$ the line of slope $-1/b$ tangent to $\Lambda$. Then $\Lambda$ is minimal for the ellipsoid $E(bc,c)$ if and only if $\Lambda$ is maximal (cf. Lemma \ref{lemma: convex path})  
under $\eta$.
\end{proposition}

{We first present a proof of Propositon \ref{prop: classification} relying on ECH capacities, after which we will present a purely combinatorial proof.}

\begin{proof}[ECH capacities proof of Proposition \ref{prop: classification}] Let $\Lambda$ be a convex integral lattice path that is not maximal under any line of slope $-1/b$. By Example \ref{eg: action}, the action of $\Lambda$ equals that of the maximal generator $\Lambda'$ under some such line $\eta$ tangent to both $\Lambda$ and $\Lambda'$. The path $\Lambda'$ is uniquely maximal, so $\Lambda$ must enclose fewer lattice points than $\Lambda'$ does.
 
Assume $\Lambda$ encloses $i$ fewer lattice points than $\Lambda'$. Let $j$ be the number of lattice points on $\eta$ (which all must be on $\Lambda'$). We have $i<j$ because if $i\geq j$ then the action of $\Lambda$ is strictly greater than the action of the maximal generator under the next line below $\eta$ of slope $-1/b$ passing through a lattice point in the first quadrant. This lower action is the $k^\text{th}$ ECH capacity of $E(bc,c)$ for $k=I(\Lambda')/2-j$ by \cite[Prop.~1.2]{H2}. However, if $i\geq j$ then $k\geq I(\Lambda')/2-i=I(\Lambda)/2$. Since ECH capacities are increasing, we have $c_{I(\Lambda)/2}(E(bc,c))\leq c_k(E(bc,c))$, which is strictly less than the action of $\Lambda$, and so $\Lambda$ cannot be minimal as its action does not equal $c_{I(\Lambda)/2}(E(bc,c))$.

Now we may assume $i<j$. Number the lattice points on $\eta$ from 1 to $j$ from the upper left to lower right. We cannot have $j=1$, because by $i<j$, the path $\Lambda$ could not then be tangent to $\eta$ because they do not share any lattice points. Therefore there are at least two other convex integral lattice paths with the same action as $\Lambda$ and enclosing the same number of lattice points: the convex hull of the set of points strictly enclosed by $\eta$ and including $\{1,\dots,j-i\}$ and the convex hull of the set of points strictly enclosed by $\eta$ and including $\{i+1,\dots,j\}$. At least one of these is not $\Lambda$, so $\Lambda$ cannot be the unique convex integral lattice path with its action amongst lattice paths enclosing $I(\Lambda)/2+1$ lattice points. Thus $\Lambda$ certainly cannot be the unique minimizer of its action amongst this set of lattice paths, and thus cannot be minimal.
\end{proof}


To give a purely combintorial proof of Proposition \ref{prop: classification}, we first set up some notation. Fix $b\ge 1$ and $c>0$. For an integer lattice point $Q=(x, y)$ in the first quadrant, let $A(Q)$ denote the $E(bc,c)$ \emph{action of} $Q$, 
    \[
    A(Q):=cx+bcy.
    \]
Let $\eta_Q$ denote the line of slope $-1/b$ that goes through $Q$, 
and let $\Lambda_Q$ denote the unique, purely elliptic, convex generator, maximal under $\eta_Q$, given by Lemma \ref{lemma: convex path}. Define the quantity $L(Q)=L(\Lambda_Q)$ to be the number of lattice points enclosed under $\Lambda_Q$. Lastly, write
    \[
    S=\{L(Q):Q\in \bbN^2\}\subseteq \bbN
    \]
to be the set of all indices of integer lattice points in the first quadrant $\bbN^2$. 

\begin{warning*}

We note that the definitions of $A(Q)$,  $\eta_Q$, $\Lambda_Q$, $L(Q)$ and $S$ all implicitly depend on the parameters $b\geq 1$ and $c>0$, as the domain $E(bc,c)$ is unambiguous. Additionally, when describing a convex generator $\Lambda$, we will use ``minimal" in short of ``minimal for the ellipsoid $E(bc,c)$".
\end{warning*}

The purely combinatorial idea behind Proposition \ref{prop: classification} is as follows. By definition, a minimal convex generator of a given index is unique when it exists, which prompts us to classify minimal generators by index. By \cite[Lem.~2.1(a)]{H}, each $\Lambda_Q$ is minimal, hence for $k\in S$ the classification of minimal generators of index $2k$ follows directly by noting that $I(\Lambda_Q)=2L(\Lambda_Q)=2L(Q)$. For $k\notin S$, we use tools similar to those developed in Lemma \ref{lemma: gen example} to construct two distinct purely elliptic convex generators of index $2k$ that both minimize the ellipsoid symplectic action. To do so, we interpret symplectic action in terms of the following: 

\begin{lemma}\label{lemma: classificationaction}
    Fix real numbers $b\ge 1$ and $c>0$. Given a convex generator $\Lambda$, we have 
    \begin{equation}\label{eqn: action sup}
    A_{E(bc,c)}(\Lambda)=\max_{Q \text{ \rm under } \Lambda}A(Q),
    \end{equation}
    where the maximum is taken over all lattice points $Q$ enclosed by $\Lambda$. In particular, for a lattice point $Q_0\in\bbN^2$, we have $A_{E(bc,c)}(\Lambda_{Q_0})=A(Q_0)$.
\end{lemma}

\begin{proof}

Among all lattice points $Q$ enclosed under $\Lambda$, $A(Q)$ is maximized when the line $\eta_Q$ is tangent to $\Lambda$. Therefore, \eqref{eqn: action sup} follows from Example \ref{eg: action} and the definition of $A(Q)$. The rest of the lemma follows from the observation that $\eta_Q$ is, by construction, tangent to $\Lambda_Q$.
\end{proof}

The following equivalences illustrate a close relationship between $\eta_Q$, $\Lambda_Q$, $L(Q)$, and $A(Q)$.

\begin{lemma}\label{lemma: classificationequiv}

Fix real numbers $b\ge1, c>0$, and let $Q_1$, $Q_2\in \bbN^2$ be two lattice points. The following are equivalent:
    \begin{enumerate}
    \item $A(Q_1)=A(Q_2)$,
    
    \item $\eta_{Q_1}=\eta_{Q_2}$,
    
    \item $\Lambda_{Q_1}=\Lambda_{Q_2}$,
    
    \item $L(Q_1)=L(Q_2)$.
    \end{enumerate}
Further, $A(Q_1)<A(Q_2)$ precisely when $L(Q_1)<L(Q_2)$, which happens precisely when $\eta_{Q_1}$ is strictly below the line $\eta_{Q_2}$.
\end{lemma}

\begin{proof}

Note that the quantity $A(Q)$ is a multiple of the $y$-intercept of $\eta_Q$ by the constant $b\cdot c$, hence (1) and (2) are equivalent. By the uniqueness statement of Lemma \ref{lemma: convex path}, (2) implies (3). By  definition, (3) implies (4).

We now argue that (4) implies (2). If $\eta_{Q_1}\ne \eta_{Q_2}$, then since they are lines of the same slope, one must have a larger $y$-intercept, say  $\eta_{Q_1}$. Now, since by definition they go through a lattice point, the maximal convex generator $\Lambda_{Q_1}$ encloses strictly more lattice points than $\Lambda_{Q_2}$, which shows that $L(Q_1)>L(Q_2)$. The last statement of the lemma follows similarly from definition.
\end{proof}

We can now consider the base when $b$ is irrational.

\begin{lemma}\label{lemma: classificationirrational}

Fix real numbers $b\ge1$ and $c>0$. The set $S=\bbN$ if and only if $b$ is irrational.
\end{lemma}

\begin{proof}

If $b=p/q$, with $p,q\in \bbN$, select any $x\in \N$, with $x>p$, and select some $y\in\N$ large enough so that that $L(x,y)>1$. By Lemma \ref{lemma: classificationequiv}, we have $L(x,y)=L(x-p,y+q)$. In this case, we have $L(x,y)-1\notin S$. This is because if $L(x,y)-1=L(Q)$ for some $Q$, then $\eta_Q$ is strictly below the line $\eta_{(x,y)}$. But then $\eta_Q$ misses at least two lattice points $(x,y)$ and $(x-p, y+q)$. Therefore, $L(x,y)-1\notin S$, and so $S\subsetneq \bbN$ if $b\in \bbQ$.

Now suppose $b\notin \bbQ$. We cannot have $L(x,y)=L(x',y')$ for distinct $(x,y)\ne (x',y')$ since otherwise $b$ is the ratio of integers $(x-x')/(y'-y)$ by Lemma \ref{lemma: classificationequiv}. Therefore if $k\in S$, then $k-1\in S$. The claim $S=\bbN$ now follows from the fact that $S$ is unbounded, as $L(N,0)\ge N$ for any $N\in \bbN$.
\end{proof}

\begin{proof}[Combinatorial proof of Proposition \ref{prop: classification}]


The if direction is the content of \cite[Lem.~2.1(a)]{H}, we now prove the converse. Let $\Lambda$ be a minimal convex generator. We will show that $\Lambda$ is maximal under $\eta$. Let $k$ denote the integer $L(\Lambda)$. By definition, a minimal convex generator of a given index is unique when exists, therefore if $k=L(Q)\in S$ we must have $\Lambda=\Lambda_Q$, since $L(\Lambda_Q)=L(Q)=k$ and $\Lambda_Q$ is minimal. Note that combining with Lemma \ref{lemma: classificationirrational} this already proves the proposition for the case when $b$ is irrational.

It suffices to consider when $k\notin S$. For this case, we construct two distinct convex generators, $\Lambda_1$ and $\Lambda_2$, both of index $2k$, such that both minimize the ellipsoid symplectic action among all generators with index $2k$. Choose $Q$ to be a lattice point such that $L(Q)>k$ is the smallest element in $S$ greater than $k$. This can be done since $S$ is unbounded. Let $Q_1,...,Q_n\in \bbN^2$ be all of the lattice points such that $L(Q_i)=L(Q)$. Using Lemma \ref{lemma: classificationequiv}, we have $L(Q)-n< k$ by the minimality of $L(Q)>k$ in $S$. Define $m=L(Q)-k < n$. We define two orderings $\preceq_1$, $\preceq_2$ on the set of integer lattice points $\bbN^2$, in the spirit of Lemma \ref{lemma: gen example},
    \begin{align*}
    Q_1\preceq_1 Q_2 \iff& L(Q_1)< L(Q_2) \text{ or } (L(Q_1)=L(Q_2), x_1\le x_2); \\
    Q_1\preceq_2 Q_2 \iff& L(Q_1)< L(Q_2) \text{ or } (L(Q_1)=L(Q_2), x_1\ge x_2).
    \end{align*}
Both $\preceq_1,\preceq_2$ are strict total orders. Thus we may rearrange the $Q_i$ such that $Q_i\preceq_1 Q_j$ whenever $i\le j$. Now by the definition of $\preceq_2$ (and the fact that $L(Q_i)=L(Q_j)$) it follows that $Q_i\preceq_2 Q_j$ if and only if $i\ge j$. By the proof of Lemma \ref{lemma: gen example}, we obtain two purely elliptic convex generators $\Lambda_1$, $\Lambda_2$, such that $\Lambda_1$ encloses exactly all $Q\in\bbN^2$ with $Q\preceq_1 Q_{n-m}$, and $\Lambda_2$ encloses exactly all $Q\preceq_2 Q_m$. Specifically, $\Lambda_1$ is achieved as the maximal convex generator under the line $\eta_Q$ rotated by a sufficiently small amount clockwise at the point $Q_{n-m}$, and $\Lambda_2$ maximal under $\eta_Q$ rotated counterclockwise at $Q_m$, both of which can be done since the number of lattice points enclosed by $\eta_Q$ in the first quadrant is finite. Note that $\Lambda_1\ne \Lambda_2$ since $0<m<n$. Note also that by construction $L(\Lambda_1)=L(\Lambda_2)=L(Q)-m=k$, since each $\Lambda_i$ omits precisely $m$ points that are enclosed under $\eta_Q$.

We examine the symplectic action of $\Lambda_1$ and $\Lambda_2$ for the ellipsoid $E(bc,c)$. Since $A(Q_{n-m})=A(Q_m)=A(Q)$ by the equivalence in Lemma \ref{lemma: classificationequiv}, we see that by equation \eqref{eqn: action sup}, $A_{E(bc,c)}(\Lambda_1)=A_{E(bc,c)}(\Lambda_2)=A(Q)$. If $\Lambda$ is a generator with $A_{E(bc,c)}(\Lambda)< A(Q)$, then by equation \eqref{eqn: action sup} for every $Q'$ enclosed by $\Lambda$, we have $A(Q')<A(Q)$, which implies that $L(Q')<L(Q)$ by Lemma \ref{lemma: classificationequiv}. But there are at most $L(Q)-n$ many such lattice points $Q'$, hence $L(\Lambda)\le L(Q)-n<k$ follows from another application of \ref{lemma: classificationequiv}. This shows that $\Lambda_1\ne \Lambda_2$ both minimize symplectic action, hence there exists no minimal convex generator with index $2k$, $k\notin S$. The proof is now complete.
\end{proof}

We now use the result of Proposition \ref{prop: classification} to provide a concrete description of a minimal convex generator in terms of its elliptic orbit decomposition.

\begin{proposition}\label{prop: p=2classification}
Let $p > 2$ be an odd integer and $c > 0$, then all minimal convex generators of $E(pc/2, c)$ are of the form 
    \[
    e_{1, 0}^k e_{\frac{p+1}{2}, 1}^{m_1} e_{p, 2}^d e_{\frac{p-1}{2}, 1}^{m_2},
    \]
with the following constraints: $m_i\in\{0,1\}, d \ge 0$, $0 \le k < \frac{p-1}{2}$ if $m_1 = 1$, and $0 \le k < \frac{p+1}{2}$ if $m_1 = 0$. 
\end{proposition}
\begin{proof}
Let $\Lambda$ be a minimal convex generator for $E(pc/2, c)$. By Proposition \ref{prop: classification}, there is a line $\eta$ of slope $-2/p$ that is tangent to $\Lambda$ such that $\Lambda$ is maximal under $\eta$. Recall from Definition \ref{def:tangency} the notion of tangency, thus $\eta$ touches $\Lambda$, and $\Lambda$ lies entirely in the closed half plane to the lower left of $\eta$. In particular, $\Lambda$ contains at least one integer lattice point $Q = (q_1, q_2)$ in the first quadrant. 

Suppose $\Lambda$ has an edge $e_{i, j}^l$ whose right endpoint is $Q$, where $i, j \ge 0$ are coprime and $l > 0$. In this case, we must have $j/i \le 2/p$. We discuss the following two cases:
\begin{enumerate}
\item If $i \ge p$, then the point $(q_1 - p, q_2 + 2)$ on $\eta$ is in the first quadrant and therefore must be enclosed by $\Lambda$. This forces $j/i = 2/p$ and hence $j = 2, i = p$ as they are coprime. In this case, $e_{i, j}^l = e_{p, 2}^l$. 
\item If $i < p$, then $j/i < 2/p$ and $j < 2$. Then if $j = 1$, we have $i \ge (p+1)/2$, so $\Lambda$ must enclose $(q_1 - (p+1)/2, q_2 + 1)$, which is to the lower left of $\eta$. This forces $j/i \ge 2/(p+1)$, so $i = (p+1)/2$. Now suppose $l \ge 2$, then $il > p$, so $(q_1 - p, q_2+2)$ is in the first quadrant yet not enclosed by $\Lambda$, contradicting its maximality under $\eta$. Therefore, $e_{i, j}^l = e_{\frac{p+1}{2}, 1}$.

Alternatively, if $j = 0$, then $i = 1$. Suppose $l \ge (p+1)/2$, then $(q_1 - (p+1)/2, q_2 + 1)$ in the first quadrant is not enclosed by $\Lambda$, again contradicting its maximality under $\eta$. Therefore, we must have $e_{i, j}^l = e_{1, 0}^l$ where $l < (p+1)/2$. Note that this is exactly when $m_1 = 0$. 
\end{enumerate}

On the other hand, suppose $\Lambda$ has an edge $e_{i, j}^l$ whose left endpoint is $Q$, where $i, j \ge 0$ are coprime and $l > 0$. In this case, we must have $j \ge 1$ and $j/i \ge 2/p$. We discuss the following two cases:
\begin{enumerate}
\item If $i \ge p$, then the point $(q_1 + p, q_2 - 2)$ on $\eta$ is in the first quadrant therefore must be enclosed by $\Lambda$. This forces $j/i = 2/p$, so $j = 2$ and $i = p$. We again obtain $e_{i, j}^l = e_{p, 2}^l$. 

\item If $i < p$, then $2/(p-1) \ge j/i > 2/p$, since $\Lambda$ must enclose $(q_1 + (p-1)/2, q_2 - 1)$, which is to the lower left of $\eta$. Solving the inequality gives $i = (p-1)/2, j = 1$. Now suppose $l \ge 2$, then $jl \ge 2$, so $(q_1 + p, q_2 - 2)$ is in the first quadrant yet not enclosed by $\Lambda$, a contradiction. Therefore, $e_{i, j}^l = e_{\frac{p-1}{2}, 1}$. 

\end{enumerate}

Now, suppose $\Lambda$ has an edge $e_{\frac{p+1}{2}, 1}$ with right endpoint $Q$ and another edge $e_{i, j}^l$ incident with $e_{\frac{p+1}{2}, 1}$ on the left. Then since $\Lambda$ fails to enclose $(q_1 - p, q_2 + 2)$, we need $jl < 2 - 1 = 1$ and $il < p - (p+1)/2 = (p-1)/2$. Therefore, we must have $e_{i, j}^l = e_{1, 0}^l$, where $l < (p-1)/2$. 

Finally, suppose $\Lambda$ has an edge $e_{\frac{p-1}{2}, 1}$ with right endpoint $Q$ and another edge $e_{i, j}^l$ incident with $e_{\frac{p-1}{2}, 1}$ on the right. Then since $\Lambda$ fails to enclose $(q_1 + p, q_2 - 2)$, we need $jl < 1$. That is, $j = 0$, but this is impossible because of the convexity of $\Lambda$. 

We can now conclude that $\Lambda$ is indeed of the form $e_{1, 0}^k e_{\frac{p+1}{2}, 1}^{m_1} e_{p, 2}^d e_{\frac{p-1}{2}, 1}^{m_2}$, where $m_i \in \{0, 1\}, d \ge 0$, $0 \le k < \frac{p-1}{2}$ if $m_1 = 1$, and $0 \le k < \frac{p+1}{2}$ if $m_1 = 0$. 
\end{proof}



\begin{remark}\label{rem:othergens}
For the symplectic embedding problems of the polydisk into half-integer or rational ellipsoids, it is unclear how the Hutchings criterion (Theorem \ref{thm: hu criterion}) will perform when applied to the minimal convex generators other than $e_{p, 2}^{d}$, as classified in Proposition \ref{prop: classification} and \ref{prop: p=2classification}. We still expect that the minimal generators $e_{p,2}^d$ for the ellipsoid will provide the best obstruction. This is due to the fact that among the classes of generators provided by Proposition \ref{prop: classification}, $e_{p, 2}^d$ maximizes the $x$ and $y$-intercepts for a fixed ellipsoid symplectic action. Thus, it imposes the tightest constraint to obstruct embedding in light of the first and central condition of the Hutchings criterion (cf. Definition \ref{defn: Jcurve}).  This makes $e_{p,2}^d$ an ideal candidate to study obstructions using the Hutchings criterion. However, the full combinatorial implications of the Hutchings criterion are not understood well enough to prove that for the ellipsoid, $e_{p,2}^d$ provides the best obstruction.
We now point out the directions for future study and explain the potential difficulties:

\begin{itemize}
    \item One can ask whether the new minimal generators in Proposition \ref{prop: p=2classification} will provide any obstructions via the Hutchings criterion, and how those obstructions compare to the main Theorem \ref{thm: NY1} of this paper. Since these generators have more complex product decomposition, the methods that we used in Section \ref{bigsection2} do not immediately generalize. A good understanding of the ECH index of products of arbitrary convex generators is therefore required to understand the implications of the third condition in the Hutchings criterion.
    
    \item One can ask whether the Hutchings criterion is limited using these new minimal generators in the same way as it was using $e_{p,2}^d$ by Propositions \ref{prop: example1}, \ref{prop: example2}, and \ref{prop: example3}. We note that the examples provided in Proposition \ref{prop: example1}-\ref{prop: example3} satisfy the Hutchings criterion (thus showing non-obstruction) by the trivial factorization. This method circumvents the need to discuss product decomposition of convex generators, which is inherently difficult. Due to the higher multiplicity of the new generators in Proposition \ref{prop: p=2classification}, the main tools that we developed in Section \ref{sec:arbind} no longer apply, thus one again needs to confront the intricacies of product decompositions for arbitrary convex generators.
\end{itemize}

\end{remark}



\bigskip	
\noindent \textsc{Leo Digiosia}, \texttt{digiosia@rice.edu}\\
	\textsc{Rice University} \\

\noindent \textsc{Jo Nelson}, \texttt{jo.nelson@rice.edu}\\
	\textsc{Rice University} \\

\noindent \textsc{Haoming Ning}, \texttt{hning99@uw.edu}\\
	\textsc{University of Washington} \\

\noindent \textsc{Morgan Weiler}, \texttt{mw795@cornell.edu}\\
	\textsc{Cornell University} \\

\noindent \textsc{Yirong Yang}, \texttt{yyang1@uw.edu} \\
	\textsc{University of Washington}

\end{document}

%% file: lemmacount2.pdf_tex
\begingroup%
  \makeatletter%
  \providecommand\color[2][]{%
    \errmessage{(Inkscape) Color is used for the text in Inkscape, but the package 'color.sty' is not loaded}%
    \renewcommand\color[2][]{}%
  }%
  \providecommand\transparent[1]{%
    \errmessage{(Inkscape) Transparency is used (non-zero) for the text in Inkscape, but the package 'transparent.sty' is not loaded}%
    \renewcommand\transparent[1]{}%
  }%
  \providecommand\rotatebox[2]{#2}%
  \newcommand*\fsize{\dimexpr\f@size pt\relax}%
  \newcommand*\lineheight[1]{\fontsize{\fsize}{#1\fsize}\selectfont}%
  \ifx\svgwidth\undefined%
    \setlength{\unitlength}{432.82910445bp}%
    \ifx\svgscale\undefined%
      \relax%
    \else%
      \setlength{\unitlength}{\unitlength * \real{\svgscale}}%
    \fi%
  \else%
    \setlength{\unitlength}{\svgwidth}%
  \fi%
  \global\let\svgwidth\undefined%
  \global\let\svgscale\undefined%
  \makeatother%
  \begin{picture}(1,0.41231733)%
    \lineheight{1}%
    \setlength\tabcolsep{0pt}%
    \put(0,0){\includegraphics[width=\unitlength,page=1]{lemmacount2.pdf}}%
    \put(0.00654981,0.39276027){\color[rgb]{0,0,0}\makebox(0,0)[lt]{\lineheight{1.25}\smash{\begin{tabular}[t]{l}$y$\end{tabular}}}}%
    \put(0.95438347,0.00461654){\color[rgb]{0,0,0}\makebox(0,0)[lt]{\lineheight{1.25}\smash{\begin{tabular}[t]{l}$x$\end{tabular}}}}%
    \put(-0.00885459,0.24770431){\color[rgb]{0,0,0}\makebox(0,0)[lt]{\lineheight{1.25}\smash{\begin{tabular}[t]{l}$qd$\end{tabular}}}}%
    \put(0.65827036,0.00164708){\color[rgb]{0,0,0}\makebox(0,0)[lt]{\lineheight{1.25}\smash{\begin{tabular}[t]{l}$pd$\end{tabular}}}}%
    \put(0.2269949,0.11947651){\color[rgb]{0,0,0}\makebox(0,0)[lt]{\lineheight{1.25}\smash{\begin{tabular}[t]{l}$e_{p,q}^d$\end{tabular}}}}%
    \put(0.57528475,0.29448764){\color[rgb]{0,0,0}\makebox(0,0)[lt]{\lineheight{1.25}\smash{\begin{tabular}[t]{l}$e_{1,0}^{pd}e_{0,1}^{qd}$\end{tabular}}}}%
  \end{picture}%
\endgroup%

%% file: lemmacount1.pdf_tex
\begingroup%
  \makeatletter%
  \providecommand\color[2][]{%
    \errmessage{(Inkscape) Color is used for the text in Inkscape, but the package 'color.sty' is not loaded}%
    \renewcommand\color[2][]{}%
  }%
  \providecommand\transparent[1]{%
    \errmessage{(Inkscape) Transparency is used (non-zero) for the text in Inkscape, but the package 'transparent.sty' is not loaded}%
    \renewcommand\transparent[1]{}%
  }%
  \providecommand\rotatebox[2]{#2}%
  \newcommand*\fsize{\dimexpr\f@size pt\relax}%
  \newcommand*\lineheight[1]{\fontsize{\fsize}{#1\fsize}\selectfont}%
  \ifx\svgwidth\undefined%
    \setlength{\unitlength}{438.07909868bp}%
    \ifx\svgscale\undefined%
      \relax%
    \else%
      \setlength{\unitlength}{\unitlength * \real{\svgscale}}%
    \fi%
  \else%
    \setlength{\unitlength}{\svgwidth}%
  \fi%
  \global\let\svgwidth\undefined%
  \global\let\svgscale\undefined%
  \makeatother%
  \begin{picture}(1,0.44441518)%
    \lineheight{1}%
    \setlength\tabcolsep{0pt}%
    \put(0,0){\includegraphics[width=\unitlength,page=1]{lemmacount1.pdf}}%
    \put(0.01845546,0.42509252){\color[rgb]{0,0,0}\makebox(0,0)[lt]{\lineheight{1.25}\smash{\begin{tabular}[t]{l}$y$\end{tabular}}}}%
    \put(0.95493016,0.04160036){\color[rgb]{0,0,0}\makebox(0,0)[lt]{\lineheight{1.25}\smash{\begin{tabular}[t]{l}$x$\end{tabular}}}}%
    \put(-0.00190039,0.17562818){\color[rgb]{0,0,0}\makebox(0,0)[lt]{\lineheight{1.25}\smash{\begin{tabular}[t]{l}$d$\end{tabular}}}}%
    \put(0.25147848,0.00442605){\color[rgb]{0,0,0}\makebox(0,0)[lt]{\lineheight{1.25}\smash{\begin{tabular}[t]{l}$kd$\end{tabular}}}}%
    \put(0.66236328,0.00442605){\color[rgb]{0,0,0}\makebox(0,0)[lt]{\lineheight{1.25}\smash{\begin{tabular}[t]{l}$md$\end{tabular}}}}%
    \put(0.53757414,0.27492526){\color[rgb]{0,0,0}\makebox(0,0)[lt]{\lineheight{1.25}\smash{\begin{tabular}[t]{l}$e_{1,0}^{kd}e_{m,1}^d$\end{tabular}}}}%
    \put(0,0){\includegraphics[width=\unitlength,page=2]{lemmacount1.pdf}}%
  \end{picture}%
\endgroup%

%% file: lemmaconvexity.pdf_tex
\begingroup%
  \makeatletter%
  \providecommand\color[2][]{%
    \errmessage{(Inkscape) Color is used for the text in Inkscape, but the package 'color.sty' is not loaded}%
    \renewcommand\color[2][]{}%
  }%
  \providecommand\transparent[1]{%
    \errmessage{(Inkscape) Transparency is used (non-zero) for the text in Inkscape, but the package 'transparent.sty' is not loaded}%
    \renewcommand\transparent[1]{}%
  }%
  \providecommand\rotatebox[2]{#2}%
  \newcommand*\fsize{\dimexpr\f@size pt\relax}%
  \newcommand*\lineheight[1]{\fontsize{\fsize}{#1\fsize}\selectfont}%
  \ifx\svgwidth\undefined%
    \setlength{\unitlength}{465.16165185bp}%
    \ifx\svgscale\undefined%
      \relax%
    \else%
      \setlength{\unitlength}{\unitlength * \real{\svgscale}}%
    \fi%
  \else%
    \setlength{\unitlength}{\svgwidth}%
  \fi%
  \global\let\svgwidth\undefined%
  \global\let\svgscale\undefined%
  \makeatother%
  \begin{picture}(1,0.63344387)%
    \lineheight{1}%
    \setlength\tabcolsep{0pt}%
    \put(0,0){\includegraphics[width=\unitlength,page=1]{lemmaconvexity.pdf}}%
    \put(0.00143494,0.61524621){\color[rgb]{0,0,0}\makebox(0,0)[lt]{\lineheight{1.25}\smash{\begin{tabular}[t]{l}$y$\end{tabular}}}}%
    \put(0.96077886,0.00094372){\color[rgb]{0,0,0}\makebox(0,0)[lt]{\lineheight{1.25}\smash{\begin{tabular}[t]{l}$x$\end{tabular}}}}%
    \put(-0.02758723,0.468523){\color[rgb]{0,0,0}\makebox(0,0)[lt]{\lineheight{1.25}\smash{\begin{tabular}[t]{l}$y_0$\end{tabular}}}}%
    \put(0.69796733,0.00094372){\color[rgb]{0,0,0}\makebox(0,0)[lt]{\lineheight{1.25}\smash{\begin{tabular}[t]{l}$x_0$\end{tabular}}}}%
    \put(0.26908385,0.22183425){\color[rgb]{0,0,0}\makebox(0,0)[lt]{\lineheight{1.25}\smash{\begin{tabular}[t]{l}$e_{x_0, y_0}$\\\end{tabular}}}}%
    \put(0.55285607,0.37017011){\color[rgb]{0,0,0}\makebox(0,0)[lt]{\lineheight{1.25}\smash{\begin{tabular}[t]{l}$\Lambda$\end{tabular}}}}%
    \put(0.54801924,0.50721926){\color[rgb]{0,0,0}\makebox(0,0)[lt]{\lineheight{1.25}\smash{\begin{tabular}[t]{l}$e_{1,0}^{x_0}e_{0,1}^{y_0}$\end{tabular}}}}%
  \end{picture}%
\endgroup%

%% file: lemmacommonfactor.pdf_tex
\begingroup%
  \makeatletter%
  \providecommand\color[2][]{%
    \errmessage{(Inkscape) Color is used for the text in Inkscape, but the package 'color.sty' is not loaded}%
    \renewcommand\color[2][]{}%
  }%
  \providecommand\transparent[1]{%
    \errmessage{(Inkscape) Transparency is used (non-zero) for the text in Inkscape, but the package 'transparent.sty' is not loaded}%
    \renewcommand\transparent[1]{}%
  }%
  \providecommand\rotatebox[2]{#2}%
  \newcommand*\fsize{\dimexpr\f@size pt\relax}%
  \newcommand*\lineheight[1]{\fontsize{\fsize}{#1\fsize}\selectfont}%
  \ifx\svgwidth\undefined%
    \setlength{\unitlength}{507.16164897bp}%
    \ifx\svgscale\undefined%
      \relax%
    \else%
      \setlength{\unitlength}{\unitlength * \real{\svgscale}}%
    \fi%
  \else%
    \setlength{\unitlength}{\svgwidth}%
  \fi%
  \global\let\svgwidth\undefined%
  \global\let\svgscale\undefined%
  \makeatother%
  \begin{picture}(1,0.58394355)%
    \lineheight{1}%
    \setlength\tabcolsep{0pt}%
    \put(0,0){\includegraphics[width=\unitlength,page=1]{lemmacommonfactor.pdf}}%
    \put(0.08117228,0.56725291){\color[rgb]{0,0,0}\makebox(0,0)[lt]{\lineheight{1.25}\smash{\begin{tabular}[t]{l}$y$\end{tabular}}}}%
    \put(0.96106926,0.00382316){\color[rgb]{0,0,0}\makebox(0,0)[lt]{\lineheight{1.25}\smash{\begin{tabular}[t]{l}$x$\end{tabular}}}}%
    \put(0.05159591,0.43268041){\color[rgb]{0,0,0}\makebox(0,0)[lt]{\lineheight{1.25}\smash{\begin{tabular}[t]{l}$y_0$\end{tabular}}}}%
    \put(0.72002215,0.00382316){\color[rgb]{0,0,0}\makebox(0,0)[lt]{\lineheight{1.25}\smash{\begin{tabular}[t]{l}$x_0$\end{tabular}}}}%
    \put(-0.00164153,0.38240079){\color[rgb]{0,0,0}\makebox(0,0)[lt]{\lineheight{1.25}\smash{\begin{tabular}[t]{l}$y_0-1$\end{tabular}}}}%
    \put(0.53960592,0.25078691){\color[rgb]{0,0,0}\makebox(0,0)[lt]{\lineheight{1.25}\smash{\begin{tabular}[t]{l}$\Lambda$\end{tabular}}}}%
    \put(0.58249184,0.4267653){\color[rgb]{0,0,0}\makebox(0,0)[lt]{\lineheight{1.25}\smash{\begin{tabular}[t]{l}$e_{x_0,1}e_{0,1}^{y_0-1}$\end{tabular}}}}%
  \end{picture}%
\endgroup%

%% file: triv2.pdf_tex
\begingroup%
  \makeatletter%
  \providecommand\color[2][]{%
    \errmessage{(Inkscape) Color is used for the text in Inkscape, but the package 'color.sty' is not loaded}%
    \renewcommand\color[2][]{}%
  }%
  \providecommand\transparent[1]{%
    \errmessage{(Inkscape) Transparency is used (non-zero) for the text in Inkscape, but the package 'transparent.sty' is not loaded}%
    \renewcommand\transparent[1]{}%
  }%
  \providecommand\rotatebox[2]{#2}%
  \newcommand*\fsize{\dimexpr\f@size pt\relax}%
  \newcommand*\lineheight[1]{\fontsize{\fsize}{#1\fsize}\selectfont}%
  \ifx\svgwidth\undefined%
    \setlength{\unitlength}{477.40050789bp}%
    \ifx\svgscale\undefined%
      \relax%
    \else%
      \setlength{\unitlength}{\unitlength * \real{\svgscale}}%
    \fi%
  \else%
    \setlength{\unitlength}{\svgwidth}%
  \fi%
  \global\let\svgwidth\undefined%
  \global\let\svgscale\undefined%
  \makeatother%
  \begin{picture}(1,0.57658419)%
    \lineheight{1}%
    \setlength\tabcolsep{0pt}%
    \put(0,0){\includegraphics[width=\unitlength,page=1]{triv2.pdf}}%
    \put(0.08610575,0.54207861){\color[rgb]{0,0,0}\makebox(0,0)[lt]{\lineheight{1.25}\smash{\begin{tabular}[t]{l}$y$\end{tabular}}}}%
    \put(0.04397109,0.36117144){\color[rgb]{0,0,0}\makebox(0,0)[lt]{\lineheight{1.25}\smash{\begin{tabular}[t]{l}$2d_0$\end{tabular}}}}%
    \put(0.06119987,0.21101775){\color[rgb]{0,0,0}\makebox(0,0)[lt]{\lineheight{1.25}\smash{\begin{tabular}[t]{l}$d_0$\end{tabular}}}}%
    \put(0.26391201,0.32975147){\color[rgb]{0,0,0}\makebox(0,0)[lt]{\lineheight{1.25}\smash{\begin{tabular}[t]{l}$e_{p,2}^{d_0}$\end{tabular}}}}%
    \put(0.65038102,0.01555062){\color[rgb]{0,0,0}\makebox(0,0)[lt]{\lineheight{1.25}\smash{\begin{tabular}[t]{l}$pd_0$\end{tabular}}}}%
    \put(0.94572999,0.01555062){\color[rgb]{0,0,0}\makebox(0,0)[lt]{\lineheight{1.25}\smash{\begin{tabular}[t]{l}$x$\end{tabular}}}}%
    \put(-0.00195912,0.28721082){\color[rgb]{0,0,0}\makebox(0,0)[lt]{\lineheight{1.25}\smash{\begin{tabular}[t]{l}$y(\Lambda)$\end{tabular}}}}%
    \put(0.70183314,0.13995629){\color[rgb]{0,0,0}\makebox(0,0)[lt]{\lineheight{1.25}\smash{\begin{tabular}[t]{l}$\Lambda$\end{tabular}}}}%
    \put(0.73641633,0.00757368){\color[rgb]{0,0,0}\makebox(0,0)[lt]{\lineheight{1.25}\smash{\begin{tabular}[t]{l}$x(\Lambda)$\end{tabular}}}}%
  \end{picture}%
\endgroup%

%% file: triv1.pdf_tex
\begingroup%
  \makeatletter%
  \providecommand\color[2][]{%
    \errmessage{(Inkscape) Color is used for the text in Inkscape, but the package 'color.sty' is not loaded}%
    \renewcommand\color[2][]{}%
  }%
  \providecommand\transparent[1]{%
    \errmessage{(Inkscape) Transparency is used (non-zero) for the text in Inkscape, but the package 'transparent.sty' is not loaded}%
    \renewcommand\transparent[1]{}%
  }%
  \providecommand\rotatebox[2]{#2}%
  \newcommand*\fsize{\dimexpr\f@size pt\relax}%
  \newcommand*\lineheight[1]{\fontsize{\fsize}{#1\fsize}\selectfont}%
  \ifx\svgwidth\undefined%
    \setlength{\unitlength}{487.90049636bp}%
    \ifx\svgscale\undefined%
      \relax%
    \else%
      \setlength{\unitlength}{\unitlength * \real{\svgscale}}%
    \fi%
  \else%
    \setlength{\unitlength}{\svgwidth}%
  \fi%
  \global\let\svgwidth\undefined%
  \global\let\svgscale\undefined%
  \makeatother%
  \begin{picture}(1,0.56417566)%
    \lineheight{1}%
    \setlength\tabcolsep{0pt}%
    \put(0,0){\includegraphics[width=\unitlength,page=1]{triv1.pdf}}%
    \put(0.08425269,0.53041266){\color[rgb]{0,0,0}\makebox(0,0)[lt]{\lineheight{1.25}\smash{\begin{tabular}[t]{l}$y$\end{tabular}}}}%
    \put(0.0430248,0.35339876){\color[rgb]{0,0,0}\makebox(0,0)[lt]{\lineheight{1.25}\smash{\begin{tabular}[t]{l}$2d_0$\end{tabular}}}}%
    \put(0.0598828,0.20647648){\color[rgb]{0,0,0}\makebox(0,0)[lt]{\lineheight{1.25}\smash{\begin{tabular}[t]{l}$d_0$\end{tabular}}}}%
    \put(0.40580323,0.2888367){\color[rgb]{0,0,0}\makebox(0,0)[lt]{\lineheight{1.25}\smash{\begin{tabular}[t]{l}$e_{p,2}^{d_0}$\end{tabular}}}}%
    \put(0.63945872,0.01521596){\color[rgb]{0,0,0}\makebox(0,0)[lt]{\lineheight{1.25}\smash{\begin{tabular}[t]{l}$pd_0$\end{tabular}}}}%
    \put(0.94689791,0.01521596){\color[rgb]{0,0,0}\makebox(0,0)[lt]{\lineheight{1.25}\smash{\begin{tabular}[t]{l}$x$\end{tabular}}}}%
    \put(-0.00191695,0.14575739){\color[rgb]{0,0,0}\makebox(0,0)[lt]{\lineheight{1.25}\smash{\begin{tabular}[t]{l}$y(\Lambda)$\end{tabular}}}}%
    \put(0.25940665,0.09964179){\color[rgb]{0,0,0}\makebox(0,0)[lt]{\lineheight{1.25}\smash{\begin{tabular}[t]{l}$\Lambda$\end{tabular}}}}%
    \put(0.77898188,0.00741069){\color[rgb]{0,0,0}\makebox(0,0)[lt]{\lineheight{1.25}\smash{\begin{tabular}[t]{l}$x(\Lambda)$\end{tabular}}}}%
  \end{picture}%
\endgroup%

%% file: section3.pdf_tex
\begingroup%
  \makeatletter%
  \providecommand\color[2][]{%
    \errmessage{(Inkscape) Color is used for the text in Inkscape, but the package 'color.sty' is not loaded}%
    \renewcommand\color[2][]{}%
  }%
  \providecommand\transparent[1]{%
    \errmessage{(Inkscape) Transparency is used (non-zero) for the text in Inkscape, but the package 'transparent.sty' is not loaded}%
    \renewcommand\transparent[1]{}%
  }%
  \providecommand\rotatebox[2]{#2}%
  \newcommand*\fsize{\dimexpr\f@size pt\relax}%
  \newcommand*\lineheight[1]{\fontsize{\fsize}{#1\fsize}\selectfont}%
  \ifx\svgwidth\undefined%
    \setlength{\unitlength}{335.86020402bp}%
    \ifx\svgscale\undefined%
      \relax%
    \else%
      \setlength{\unitlength}{\unitlength * \real{\svgscale}}%
    \fi%
  \else%
    \setlength{\unitlength}{\svgwidth}%
  \fi%
  \global\let\svgwidth\undefined%
  \global\let\svgscale\undefined%
  \makeatother%
  \begin{picture}(1,0.63200823)%
    \lineheight{1}%
    \setlength\tabcolsep{0pt}%
    \put(0,0){\includegraphics[width=\unitlength,page=1]{section3.pdf}}%
    \put(0.94655765,0.00524832){\makebox(0,0)[lt]{\lineheight{1.25}\smash{\begin{tabular}[t]{l}$x$\end{tabular}}}}%
    \put(-0.00225342,0.60909603){\makebox(0,0)[lt]{\lineheight{1.25}\smash{\begin{tabular}[t]{l}$y$\end{tabular}}}}%
    \put(0.3828687,0.32725671){\makebox(0,0)[lt]{\lineheight{1.25}\smash{\begin{tabular}[t]{l}$(4,3)$\end{tabular}}}}%
    \put(0.7641808,0.11938426){\makebox(0,0)[lt]{\lineheight{1.25}\smash{\begin{tabular}[t]{l}$\eta(4,3)$\end{tabular}}}}%
    \put(0.21963146,0.44968463){\color[rgb]{0,0,1}\makebox(0,0)[lt]{\lineheight{1.25}\smash{\begin{tabular}[t]{l}$\eta'$\end{tabular}}}}%
    \put(-0.00225342,0.37685633){\makebox(0,0)[lt]{\lineheight{1.25}\smash{\begin{tabular}[t]{l}$4$\end{tabular}}}}%
    \put(0.71431866,0.00524832){\makebox(0,0)[lt]{\lineheight{1.25}\smash{\begin{tabular}[t]{l}$8$\end{tabular}}}}%
  \end{picture}%
\endgroup%